\documentclass[a4paper,11pt,reqno]{amsart}

\usepackage[margin=1in,includehead,includefoot,asymmetric,marginparwidth=1cm]{geometry}
\usepackage[utf8]{inputenc}
\usepackage[T1]{fontenc}
\usepackage{lmodern,microtype}
\usepackage{hyperref}
\usepackage{url}
\usepackage{booktabs}
\usepackage[table]{xcolor}
\usepackage{mathtools,amssymb,amsthm}
\usepackage[foot]{amsaddr}
\usepackage{graphicx}
\usepackage{enumitem}
\usepackage[textsize=footnotesize,disable]{todonotes}

\graphicspath{{./graphics/}}

\newcommand{\cC}{\mathcal{C}}

\newcommand{\cM}{\mathcal{M}}
\newcommand{\cN}{\mathcal{N}}
\newcommand{\cO}{\mathcal{O}}
\newcommand{\cP}{\mathcal{P}}
\newcommand{\cT}{\mathcal{T}}
\newcommand{\cX}{\mathcal{X}}
\newcommand{\cY}{\mathcal{Y}}
\newcommand{\cZ}{\mathcal{Z}}
\DeclareMathOperator{\rev}{rev}
\DeclareMathOperator{\id}{id}
\DeclareMathOperator{\ord}{ord}

\newtheorem{theorem}{Theorem}
\newtheorem{corollary}[theorem]{Corollary}
\newtheorem{lemma}[theorem]{Lemma}

\newtheorem{question}[theorem]{Question}

\definecolor{defblue}{rgb}{0.1,0.1,0.7}
\newcommand{\defi}[1]{\textcolor{defblue}{\emph{#1}}}

\definecolor{burgundy}{rgb}{144,0,32}

\newcommand{\torsten}[1]{\todo[inline,color=green!40]{\textbf{Tt:} #1}}


\hypersetup{
  pdftitle={Hamilton paths and cycles in flip graphs of (almost-)perfect matchings},
  pdfauthor={Sofia Brenner, Justin Dallant, Linda Kleist, Robert Lauff, Torsten M\"utze and Torben Sch\"urenberg}
}

\title{Hamilton paths and cycles in flip graphs of (almost-)perfect matchings}

\author{Sofia Brenner}
\address[Sofia Brenner]{Fakult\"at für Mathematik und Informatik, Universit\"at Leipzig, Germany}
\email{sofia.brenner@uni-leipzig.de}

\author{Justin Dallant}
\address[Justin Dallant]{Fakult\"at Informatik, TU Dresden, Germany}
\email{justin.dallant@tu-dresden.de}

\author{Linda Kleist}
\address[Linda Kleist]{Fachbereich Informatik, Universit\"at Hamburg, Germany}
\email{linda.kleist@uni-hamburg.de}

\author{Robert Lauff}
\address[Robert Lauff]{Institut f\"ur Mathematik, TU Berlin, Germany}
\email{lauff@math.tu-berlin.de}

\author{\\Torsten M\"utze}
\address[Torsten M\"utze]{Institut f\"ur Mathematik, Universit\"at Kassel, Germany}
\email{tmuetze@mathematik.uni-kassel.de}

\author{Torben Sch\"urenberg}
\address[Torben Sch\"urenberg]{Zentrum f\"ur Industriemathematik, Universit\"at Bremen, Germany}
\email{torsch@uni-bremen.de}

\thanks{This paper originated at the workshop `Combinatorics, Algorithms and Geometry', which took place in Kassel in March 2026. Workshop funding from the DFG priority programme SPP 2458 `Combinatorial Synergies' and the DFG Heisenberg grant~522790373 `Combinatorial Algorithms' is gratefully acknowledged.}

\begin{document}

\begin{abstract}
We consider the set of matchings of a graph and a local change operation, called a flip, between them.
In the combinatorial setting, the base graphs are either complete graphs or complete bipartite graphs, and in the geometric setting, the graphs are embedded on point sets in the plane, with the requirement that edges must be drawn as straight lines and must not cross.
For base graphs with an even number of vertices, we consider perfect matchings, i.e., all vertices are matched, and for base graphs with an odd number of vertices, we consider almost-perfect matchings, i.e., all but one vertex of the graph are matched.
A 2-flip between two perfect matchings exchanges two edges, and a 1-flip between two almost-perfect matchings exchanges one edge.
The corresponding flip graph has the set of perfect or almost-perfect matchings as vertices, with pairs of them connected by an edge if and only if they differ in a 2-flip or 1-flip, respectively.
In this work, we provide a comprehensive picture of Hamiltonicity properties of these flip graphs, i.e., under which conditions the flip graphs admit spanning paths and cycles.

We prove that the flip graphs in the combinatorial setting are Hamilton-connected, i.e., they admit a Hamilton path between any two vertices, or, if the flip graphs are bipartite, we prove that they are Hamilton-laceable, i.e., they admit a Hamilton path between any two vertices from different partition classes.
For almost-perfect matchings in the complete graph under 1-flips this answers a problem raised by Aichholzer, Dorfer, Rieck and Verciani.

In the geometric setting, while the flip graphs are connected, we prove that any path in them misses exponentially many vertices, in particular, they have no Hamilton paths or cycles.
For points in convex position and almost-perfect matchings under 1-flips, we complement this by constructing a cycle in the flip graph that visits almost all vertices.

We also introduce a new directed variant of the combinatorial setting, where each matching edge is directed.
This gives rise to different types of 2-flips and 1-flips, and leads to a more fine-grained picture of properties of the associated flip graphs, such as number of components, bipartiteness, diameter, Hamilton-connectedness/-laceability.
\end{abstract}


\maketitle
\newpage
\section{Introduction}

Matchings in graphs are fundamental combinatorial objects, and a large body of classical results and algorithms in graph theory and optimization is devoted to them.
Think for example of Hall's marriage theorem, Tutte's 1-factor theorem, Petersen's theorem, the Hungarian algorithm, Edmonds' blossom algorithm, the Gale-Shapley algorithm, and the variety of min-max theorems that link matchings to other graph objects.

The starting point of our paper are \defi{perfect matchings} in the complete graph~$K_n$ with an even number of vertices~$n=2m$, i.e., this is an $m$-subset of edges of~$K_n$, no two of which share an end vertex.
We equip the set of these matchings with a minimum change operation, referred to as a \defi{2-flip}, which consists of replacing two edges $\{i,j\},\{k,\ell\}$ by $\{i,k\},\{j,\ell\}$ (or $\{i,\ell\},\{j,k\}$); see the left hand side of Figure~\ref{fig:flip}.
In other words, two perfect matchings of~$K_n$ differ in a 2-flip, if their symmetric difference is a 4-cycle.
In this way, we obtain a \defi{flip graph}, which has as vertices all perfect matchings, and an edge between any two matchings that differ in a 2-flip; see Figure~\ref{fig:settings}~(a).
Recall that in the \defi{perfect matching polytope}~\cite{MR371732}, two matchings are adjacent if their symmetric difference is an even cycle of arbitrary length, so our 2-flips are more restrictive and form a subset of the edges of the polytope.

\begin{figure}[h!]
\includegraphics{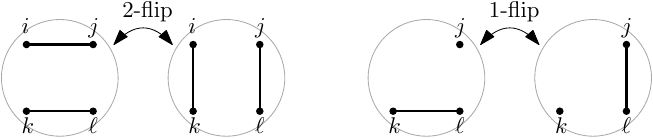}
\caption{Illustration of flip operations on perfect matchings (left) of~$K_{2m}$ and almost-perfect matchings (right) of~$K_{2m-1}$. Edges of the matchings that are not modified are not shown.}
\label{fig:flip}
\end{figure}

\begin{figure}[t!]
\includegraphics[scale=0.6]{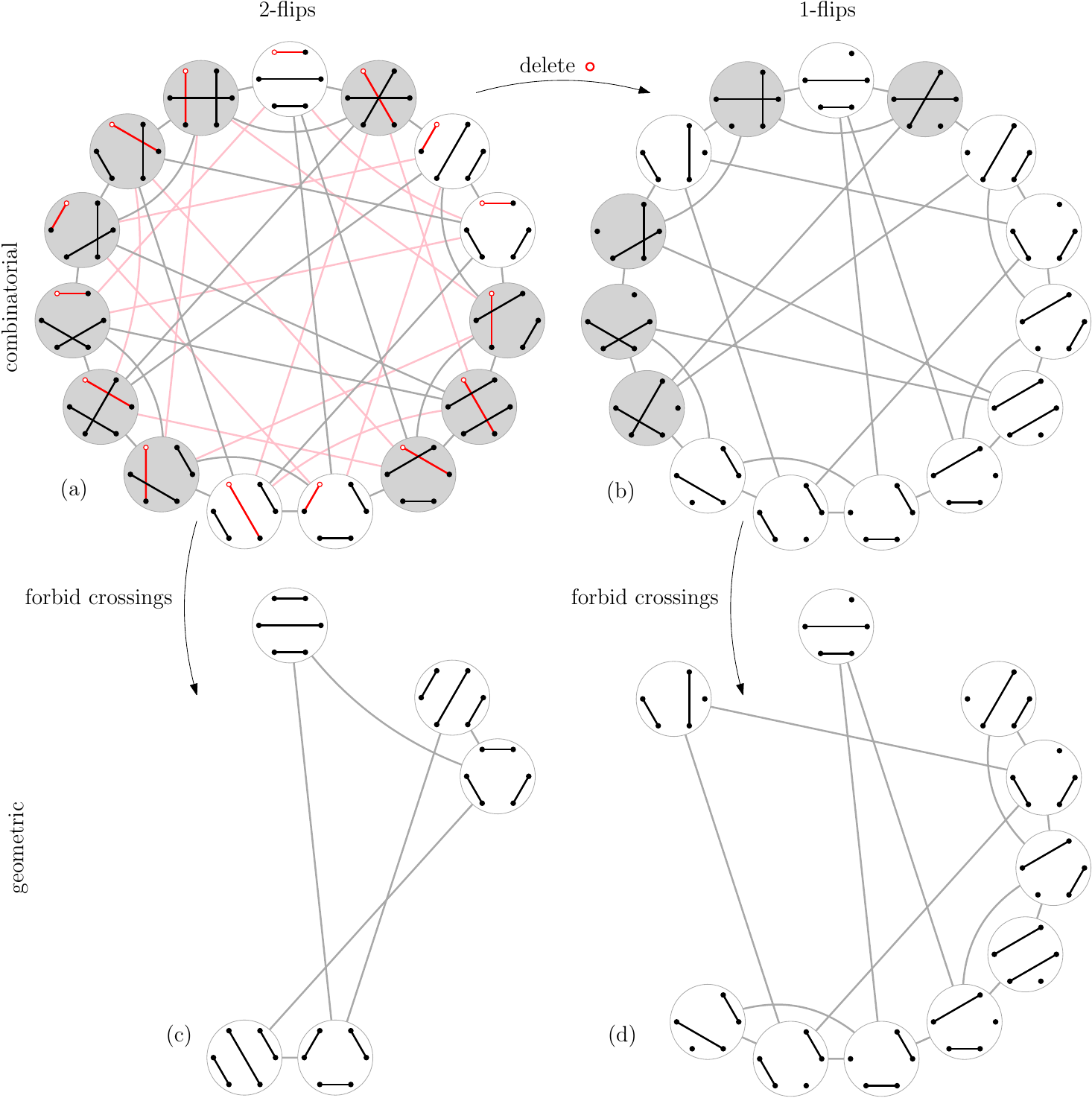}
\caption{Flip graphs on perfect and almost-perfect matchings in the different settings.
The marked vertex and its incident edge in~(a), and the corresponding flips, are deleted when going to~(b).
The shaded matchings in (a) and~(b) have crossings and are deleted when going down to the geometric setting.
The graph in~(b) is shown again in Figure~\ref{fig:K5-1flips},
the graph in~(d) is redrawn more nicely in Figure~\ref{fig:G5geom}.
}
\label{fig:settings}
\end{figure}

Our flip graph has been studied before by Diaconis and Holmes~\cite{MR1665632} in the context of phylogenetic trees, and the authors describe the construction of a \defi{Hamilton path}, i.e., a sequence of all perfect matchings, each appearing exactly once, such that any two consecutive matchings differ in a 2-flip.
Diaconis and Holmes~\cite{MR1887626} later proved mixing time results for a random walk in the flip graph.
Fabila-Monroy, Flores-Pe\~{n}aloza, Huemer, Hurtado, Urrutia and Wood~\cite{MR2535071} studied the chromatic number of the flip graph, i.e., the minimum number of colors needed so that any two neighboring vertices receive different colors.
They established the upper bound~$4m-4$, and conjectured that the true chromatic number is~$m+1$ for $m\geq 2$, which was confirmed by computer for $m=2,3,4$.
In this result, $m$ denotes the number of edges of the matching, and $n$ the number of vertices.
Interestingly, they were not able to prove a non-trivial lower bound.
Cioab\u{a}, Royle and Tan~\cite{MR4333854} later improved the upper bound for the chromatic number and investigated spectral properties of the flip graph.

We refer to the setting discussed so far with an abstract base graph (such as $K_n$) as the \defi{combinatorial} setting.
It is very natural to consider the same kind of objects and flips between them in a \defi{geometric} setting, where edges must be drawn as straight lines and must not cross.
In the simplest case, the $n=2m$ points are arranged in convex position.
Note that non-crossing perfect matchings on points in convex position are one of the classical Catalan families, counted by the Catalan numbers~$C_m=\frac{1}{m+1}\binom{2m}{m}$.
We thus obtain another flip graph, whose vertices are non-crossing matchings on the $n$ points, and two of them are connected by an edge if they differ in a 2-flip; see Figure~\ref{fig:settings}~(c).
Clearly, this flip graph is an induced subgraph of the flip graph of the combinatorial setting discussed before.
Hernando, Hurtado and Noy~\cite{MR1939072} proved that the geometric flip graph is bipartite, and that it has diameter~$m-1$ and connectivity~$m-1$.
Furthermore, they established the existence of a Hamilton cycle for even $m\geq 4$, and proved that no Hamilton path exists for odd $m\geq 5$, due to the hugely imbalanced bipartition in this case.
In fact, using a straightforward Catalan bijection from non-crossing perfect matchings to Dyck words, one can apply a result of Ruskey and Proskurowski~\cite{MR1041167} to directly obtain the existence of a Hamilton path in the flip graph, with the extra property that every 2-flip either adds or removes an edge between two consecutive points on the convex hull.
We mention this here, because in the following we will consider another type of flip that can also be understood as a restricted 2-flip.

Specifically, motivated by token-sliding games in grid graphs studied in~\cite{MR4189477,MR5070842}, Aichholzer, Br\"{o}tzner, Perz and Schnider~\cite{MR4879965} recently considered point sets with an odd number $n=2m+1$ of points, and \defi{almost-perfect matchings}, i.e., matchings with $m$ edges that cover all but one point (which they coined \defi{odd matchings}); see Figure~\ref{fig:settings}~(d).
A natural flip operation in this context, which we refer to as a \defi{1-flip}, is the following; see the right hand side of Figure~\ref{fig:flip}: if the point~$j$ is currently unmatched and $\{k,\ell\}$ is an edge in the matching, then we may replace the edge~$\{k,\ell\}$ by $\{j,\ell\}$ (or $\{j,k\}$).
In other words, two almost-perfect matchings differ in a 1-flip, if their symmetric difference is a 2-path.
Note that a 1-flip is uniquely specified by the vertex that becomes unmatched by the flip operation.
The authors of~\cite{MR4879965} proved that the corresponding flip graph is connected for point sets in general position, establishing an upper bound of~$\cO(n^2)$ on the diameter.
For convex position, they proved that the diameter is~$\Theta(n)$.
The setting of odd matchings was further studied in~\cite{MR4993079,MR5044462}; see Section~\ref{sec:related} below for more details.

It is very natural to consider almost-perfect matchings and 1-flips between them also in the combinatorial setting of the complete graph with an odd number of vertices; see Figure~\ref{fig:settings}~(b).
In fact, there is a straightforward bijection between perfect matchings of~$K_{2m}$ and almost-perfect matchings of~$K_{2m-1}$, which simply consists of removing one fixed vertex and the matching edge incident to it.
Under this bijection, the following correspondence between 2-flips in~$K_{2m}$ and 1-flips in~$K_{2m-1}$ becomes apparent:
1-flips in $K_{2m-1}$ correspond exactly to 2-flips in $K_{2m}$ that involve the removed vertex (specifically, the alternating 4-cycle contains the removed vertex).
Consequently, the flip graph of almost-perfect matchings of~$K_{2m-1}$ is a spanning subgraph of the flip graph of perfect matchings of~$K_{2m}$, where the edges in the subgraph correspond to restricted 2-flips that involve the removed vertex.
Furthermore, the geometric flip graph considered in~\cite{MR4879965} (Figure~\ref{fig:settings}~(d)) is an induced subgraph of the combinatorial flip graph of almost-perfect matchings of~$K_{2m-1}$ under 1-flips (Figure~\ref{fig:settings}~(c)).

Perfect matchings on an even number of vertices/points, and almost-perfect matchings on an odd number of vertices/points can be seen as \defi{maximum} matchings in both settings, and this is a unifying point of view on the aforementioned questions.

The starting point of this work is the following question raised by Aichholzer, Dorfer, Rieck and Verciani~\cite{booklet-problem} during the workshop `Combinatorics, Algorithms and Geometry' in Kassel in March, 2026:
\begin{question}
\label{quest:workshop}
Does the flip graph of almost-perfect matchings in complete graphs under 1-flips admit a Hamilton cycle?
\end{question}

\subsection{Our results}

In this paper, we tackle Question~\ref{quest:workshop} and closely related ones using a very strong notion of Hamiltonicity.
This has two advantages: It gives stronger results, and at the same time simpler proofs, because of increased flexibility when applying induction. 
Specifically, we say that a graph is \defi{Hamilton-connected} if it admits a Hamilton path between any two distinct vertices.
A bipartite graph is \defi{Hamilton-laceable} if it admits a Hamilton path between any two vertices from different partition classes.
In particular, this requires that the partition classes have the same size, and is best possible in this case.
Clearly, being Hamilton-connected is stronger than having a Hamilton cycle, which is stronger than having a Hamilton path.

We establish Hamilton-connectedness and Hamilton-laceability for a variety of flip graphs of (almost-)perfect matchings under 2-flips and 1-flips in the combinatorial and geometric setting.
In particular, we positively answer Question~\ref{quest:workshop} of Aichholzer, Dorfer, Rieck and Verciani.
Table~\ref{tab:results} summarizes our results, and provides pointers to the theorems with more details stated later.
The top part of the table is devoted to the combinatorial setting, and the bottom part to the geometric setting.
In the rest of this section, we discuss the most important results from the table in more detail.

\torsten{Is the flip graph of \cite{MR1939072} Hamilton-laceable? For the base case $m=4$ the computer says yes.}

\definecolor{myred}{rgb}{1,0.753,0.796}
\definecolor{mygreen}{rgb}{0.565,0.933,0.565}

\begin{table}[t!]
\renewcommand{\arraystretch}{1.1}
\caption{Overview of results in this paper.
In the cases where the result was established in earlier work or directly follows from it, the corresponding citation is stated in the last column.
Existence results are highlighted in green, and non-existence results are highlighted in red.
}
\label{tab:results}
\makebox[0cm]{ 
\begin{tabular}{@{}lllll@{}}
\toprule 
Base graph & Matchings & Flip & Result & Reference \\
\midrule
\multicolumn{5}{@{}l}{\bfseries\itshape Combinatorial setting} \\
\midrule
$K_n$ ($n=2m$)   & perfect & 2-flip & \cellcolor{mygreen}Ham.-connected ($m\geq 1$) & Th.~\ref{thm:CE2-hconn} \\
$K_n$ ($n=2m+1$) & almost-perfect & 1-flip & \cellcolor{mygreen}Ham.-connected ($m\geq 1$) & Th.~\ref{thm:CO1-hconn} \\
$K_n$ ($n=2m+d$) & $m$ edges & 1-flip & \cellcolor{mygreen}Ham.-connected ($m,d\geq 1$) & Th.~\ref{thm:COd-hconn} \\
$K_{n,n}$   & perfect & 2-flip & \cellcolor{mygreen}Ham.-laceable ($n\geq 2$)  & Th.~\ref{thm:CEB-lace} \cite{MR683982} \\
$K_{n,n+1}$ & almost-perfect & 1-flip & \cellcolor{mygreen}Ham.-laceable ($n\geq 3$)  & Th.~\ref{thm:COB-lace} \cite{MR683982} \\
$K_{n,n+d}$ & $n$ edges & 1-flip & \cellcolor{mygreen}Ham.-connected ($n,d\geq 2$) & Th.~\ref{thm:COB-conn} \\
\midrule
$K_n$ ($n=2m$) & directed perfect & 2-flips: & & \\
               &                      & type~i & \cellcolor{mygreen}Ham.-laceable ($m\geq 4$) & Th.~\ref{thm:CED-typei-lace} \\
               &                      & type~ii & \cellcolor{mygreen}Ham.-laceable ($m\geq 3$) & Th.~\ref{thm:CED-typeii-lace} \\
               &                      & type~iii,iv,v & \cellcolor{myred}disconnected ($m\geq 2$) & Th.~\ref{thm:CED-basic} \\
               &                      & mixed & \cellcolor{mygreen}Ham.-connected ($m\geq 2$) & Th.~\ref{thm:CED-hconn} \\
$K_n$ ($n=2m+1$) & dir.\ almost-perf.   & 1-flips: & & \\
               &                      & type~i,ii,iii,iv: & \cellcolor{myred}disconnected ($m\geq 1$) & Th.~\ref{thm:COD-basic} \\ 
               &                      & mixed & \cellcolor{mygreen}Ham.-connected ($m\geq 1$) & Th.~\ref{thm:COD-hconn} \\
\midrule
\multicolumn{5}{@{}l}{\bfseries\itshape Geometric setting (non-crossing)} \\       
\midrule
$n=2m$ points & perfect    & 2-flip & \cellcolor{mygreen}Ham.\ cycle ($m\geq 4$ even) & Th.~\ref{thm:GE-even} \cite{MR1939072} \\
convex position               &  &     & \cellcolor{myred}miss $\Theta(2^m/m^{3/2})$ ($m\geq 5$ odd) & Th.~\ref{thm:GE-odd} \cite{MR1939072} \\
$n=2m+1$ pts. & almost-perfect & 1-flip & \cellcolor{myred}miss $\Theta(2^{m/5})$ ($m\geq 34$) & Th.~\ref{thm:GO-miss} \\
convex position               &  &     & \cellcolor{mygreen}$(1-o(1))$-fraction cycle ($m\geq 1$) & Th.~\ref{thm:GO-long} \\
$n=2m+1$ pts. & almost-perfect & 1-flip & \cellcolor{myred}miss $\Theta(4^m/m^{3/2})$ ($m\geq 5$) & Th.~\ref{thm:GOP} \\
general position & & & \cellcolor{myred} & \\
\bottomrule
\end{tabular}
}
\end{table}

As a warm-up, we consider the flip graph of perfect matchings in the complete graph under 2-flips, and strengthen the result of Diaconis and Holmes~\cite{MR1665632} to Hamilton-connectedness (Theorem~\ref{thm:CE2-hconn}).
Next, we resolve Question~\ref{quest:workshop} affirmatively by proving that the flip graph of almost-perfect matchings in the complete graph under 1-flips is Hamilton-connected (Theorem~\ref{thm:CO1-hconn}).
This construction is technically much more demanding, as 1-flips are considerably more restricted than 2-flips (recall Figure~\ref{fig:settings}~(a)+(b)).
Nonetheless, the construction for Theorem~\ref{thm:CO1-hconn} can be turned into an efficient algorithm for computing a Hamilton path in the flip graph, generating each new matching in constant time on average.
An implementation of this algorithm in C++ is available for download and experimentation on the Combinatorial Object Server~\cite{cos_match}.

In addition to the complete graph, we also consider the complete bipartite graph as a base graph, and establish Hamilton-laceability of the corresponding flip graphs of perfect matchings and almost-perfect matchings under 2-flips or 1-flips, respectively (Theorems~\ref{thm:CEB-lace} and~\ref{thm:COB-lace}).
These results are based on a simple observation that connects 1-flips to so-called star transpositions in permutations, and is obtained by applying a well-known result of Tchuente~\cite{MR683982}.

Still in the combinatorial setting, we also introduce a directed variant of the problem, where each edge of the matching has one of two possible orientations.
The directions of the edges allow distinguishing five different types of 2-flips and four different types of 1-flips; see Figures~\ref{fig:dir-2flips} and~\ref{fig:dir-1flips}, respectively.
In some cases the flip graphs are not connected (Theorems~\ref{thm:CED-basic} and~\ref{thm:COD-basic}), but when they are we prove Hamiltonicity (Theorems~\ref{thm:CED-typei-lace}, \ref{thm:CED-typeii-lace}, \ref{thm:CED-hconn} and~\ref{thm:COD-hconn}).

We note that all the flip graphs in the combinatorial settings (directed and undirected) are vertex-transitive, and so our Hamiltonicity results resolve special cases of Lov{\'a}sz'~\cite{MR0263646} well-known conjecture on Hamilton paths and cycles in vertex-transitive graphs.

We now move to the geometric setting, in which finding Hamilton cycles is considerably harder, and often impossible.
We already mentioned before the result of Hernando, Hurtado and Noy~\cite{MR1939072} that the flip graph of non-crossing perfect matchings of a set of $n=2m$ points in convex position admits a Hamilton cycle if $m\geq 4$ is even, and no Hamilton path if $m\geq 5$ is odd.
In fact, their non-existence argument uses the observation that the flip graph is bipartite (for all $m$), with hugely imbalanced partition classes for odd~$m$.
It follows that for odd~$m$ any path in the flip graph misses at least as many vertices as the size difference in the partition classes (minus~1), in this case $\Theta(2^m/m^{3/2})$, i.e., exponentially many.
We establish a similar result for non-crossing almost-perfect matchings on~$n=2m+1$ points in convex position, namely that any path in the flip graph misses~$\Theta(2^{m/5})$ many vertices (Theorem~\ref{thm:GO-miss}).
Interestingly, the reason for non-Hamiltonicity is different than in the results from~\cite{MR1939072} mentioned before, in particular, the flip graph in Theorem~\ref{thm:GO-miss} is not bipartite.
Instead, in our case it is the large number of degree-2 vertices that forces excluding many vertices.
Complementing this result, we construct a cycle that visits almost all vertices of the flip graph, namely a $(1-o(1))$-fraction (Theorem~\ref{thm:GO-long}).
This is not a contradiction, because even though the number of vertices missed in Theorem~\ref{thm:GO-miss} is exponential, their fraction is only~$o(1)$, and the two $o(1)$ terms decay at a different rate.

Hamilton paths and cycles in flip graphs are often referred to as \defi{combinatorial Gray codes}~\cite{MR1491049,MR4649606} in the literature.
As exemplified before, all of the constructions presented in this paper translate straightforwardly into efficient algorithms for computing the corresponding Gray code listings of matchings.

\subsection{Related work}
\label{sec:related}

There has been a substantial amount of previous work on flips in matchings, both in the combinatorial and geometric setting, sometimes with a more general notion of flips than we consider here, as explained below.
In particular, very recently there has been a flurry of exciting new developments regarding hardness results for the algorithmic problem of computing shortest paths in flip graphs on matchings.
Given two matchings of some base graph, the goal is to compute a shortest sequence of flips to transform one matching into the other.

Bonamy, Heinrich, Kobayashi, M\"{u}hlenthaler, Bousquet, Ito, Mary and Wasa~\cite{MR4008469} considered the shortest path problem for 2-flips between two given perfect matchings from a particular class of graphs.
They proved that the problem is PSPACE-complete even for split graphs and bipartite graphs of bounded bandwidth.
On the positive side, they showed that the problem is solvable in polynomial time for interval graphs, strongly chordal graphs, outerplanar graphs and cographs.
Cardinal and Steiner~\cite{MR4872027} proved that in the perfect matching polytope of a bipartite graph, unless P=NP, there is no efficient algorithm for computing a path of constant length between two matchings at distance~2 (!).

Aichholzer, Brenner, Dorfer, Hoang, Perz, Rieck and Verciani~\cite{MR4993079} considered odd matchings, i.e., almost-perfect matchings, in base graphs with an odd number of vertices.
They give a characterization of graphs in which any two odd matchings can be transformed into each other via 1-flips, i.e., for which the corresponding flip graph is connected.
They also showed that finding a shortest sequence of 1-flips between two odd matchings is NP-hard.
Subsequently, Dorfer~\cite{MR5044462} strengthened these hardness results to inapproximability, and also proved hardness of computing the diameter of the flip graph.

Houle, Hurtado, Noy and Rivera-Campo~\cite{MR2190792} considered plane perfect matchings on general point sets, but allowed a more powerful flip operation, namely symmetric difference with an alternating crossing-free cycle of arbitrary (even) length.
They proved that the corresponding flip graph is connected, and that its diameter is at most~$n-2$, for any set of $n=2m$ points.
It is unknown whether the flip graph remains connected if we impose any bound on the length of the alternating flipping cycle, in particular for 2-flips; see~\cite{MR2455502}. 
Aichholzer et al.~\cite{MR2519379} considered an even more powerful flip operation, namely symmetric difference with any number of alternating crossing-free cycles of arbitrary (even) lengths.
They established an upper bound of $\cO(\log n)$ for the diameter of the flip graph, which almost matches an $\Omega(\log n/\log \log n)$ lower bound of Razen~\cite{razen:2008}.

Binucci, Montecchiani, Perz and Tappini~\cite{binucci-montecchiani-perz-tappini:2025} proved that for point sets in general position, finding a minimum length sequence of 2-flips between two given plane perfect matchings is NP-hard.

Apart from matchings, many other combinatorial objects and their corresponding flip graphs have been studied intensively, and the flip graphs often arise as the skeleta of high-dimensional polytopes.
A classical example is the \defi{associahedron}, the flip graph of triangulations of a convex polygon, where two triangulations are connected if they differ in exchanging a single diagonal.
Lucas~\cite{MR920505} proved that this flip graph admits a Hamilton cycle.
Lucas, Roelants van Baronaigien and Ruskey~\cite{MR1239499} provided an efficient algorithm, and a simpler Hamiltonicity proof was published by Hurtado and Noy~\cite{MR1723053}.
Huemer, Hurtado, Noy and Oma\~{n}a-Pulido~\cite{MR2510231} constructed Gray codes for non-crossing set partitions and dissections of a convex polygon, both with fixed and variable number of blocks or diagonals, respectively.
Note that dissections of a convex polygon with the largest possible number of diagonals are triangulations, so this setting generalizes what we discussed before.
In a spectacular breakthrough, Dorfer~\cite{dorfer_preprint} recently proved that computing shortest paths in the associahedron is NP-hard.

Another class of polytopes that yields flip graphs on geometric objects are the (weak and strong) \defi{rectangulotopes} of Cardinal and Pilaud~\cite{MR4833064}, for which Merino and M\"utze~\cite{MR4598046} constructed Hamilton cycles.
Here the combinatorial objects are rectangulations, i.e., subdivisions of a rectangle into smaller ones, and the flip operations are local modifications of these subdivisions.

We close this section by mentioning Kotzig's infamous \defi{perfect 1-factorization conjecture}~\cite{MR173249}, which asserts that the edges of~$K_n$, $n=2m$, can be decomposed into $n-1$ perfect matchings, such that the union of any two of them forms a Hamilton cycle.

\subsection{Outline of this paper}

In Section~\ref{sec:prelim}, we collect definitions that are used throughout this paper.
In Section~\ref{sec:comb} we prove the results in the combinatorial setting listed in the top part of Table~\ref{tab:results}.
The results for the newly introduced (combinatorial) setting of directed matchings are treated separately in Section~\ref{sec:directed}.
In Section~\ref{sec:geom}, we turn to the geometric setting, proving the results listed in the bottom part of Table~\ref{tab:results}.

\section{Preliminaries}
\label{sec:prelim}

We define $[n]\coloneq \{1,2,\ldots,n\}$ and $[i,j]\coloneq\{i,i+1,\ldots,j\}$.
For any sequence $x=(x_1,\ldots,x_n)$, we write $\rev(x)\coloneq (x_n,x_{n-1},\ldots,x_1)$ for the reverse sequence.
We refer to a cycle of length~$\ell$ or a path of length~$\ell$ as an \defi{$\ell$-cycle} or \defi{$\ell$-path}, respectively.
Recall that a 1-flip in an almost-perfect matching can be specified uniquely by describing the vertex to become unmatched through the flip.
Specifically, given an almost-perfect matching~$M$ with unmatched vertex~$j$ and some matched vertex~$k$, we write $u(M,k)$ for the matching obtained from~$M$ by replacing the edge~$\{k,\ell\}$ by the edge~$\{j,\ell\}$, where $\ell$ is the neighbor of~$k$ in~$M$.
Recall that the \defi{Catalan numbers} are defined as $C_n\coloneq \frac{1}{n+1}\binom{2n}{n}$, and they have exponential asymptotic growth $C_n=(1+o(1)) 4^n/(\sqrt{\pi}n^{3/2})$.

\section{Combinatorial setting}
\label{sec:comb}

\subsection{Perfect matchings in complete graphs}

As mentioned in the introduction, the set of perfect matchings of $K_n$, $n=2m$, is in one-to-one correspondence with the set of almost-perfect matchings of $K_{n-1}$, and the bijection is given by removing one fixed vertex.
Consequently, both families of objects have the same counting sequence, namely the double factorials, $(n-1)(n-3)\cdots 3\cdot 1=(n-1)!!$ (OEIS A001147 \cite{oeis}).
In particular, the number of matchings is always odd.
Under this bijection, 2-flips in~$K_n$ are in correspondence either with 2-flips in $K_{n-1}$ (if the removed vertex is not one of the four involved in the 2-flip) or with 1-flips in $K_{n-1}$ (if the removed vertex is one of the four involved in the 2-flip).

\begin{theorem}
\label{thm:CE2-hconn}
Let $m\geq 1$ and $n\coloneq 2m$.
The flip graph of perfect matchings of~$K_n$ under 2-flips is Hamilton-connected.
\end{theorem}

\begin{figure}[b]
\includegraphics{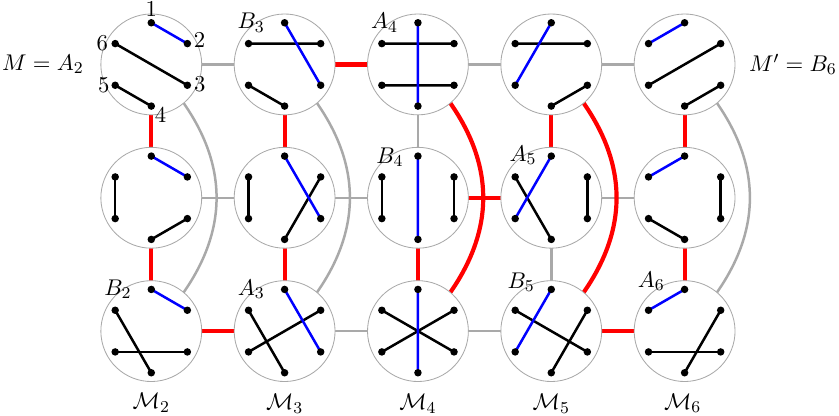}
\caption{Illustration of the proof of Theorem~\ref{thm:CE2-hconn} for $m=3$.
Edges between sets~$\cM_i$ and~$\cM_j$ for $j>i+1$ are not shown.}
\label{fig:hconn}
\end{figure}

The idea for this proof is very natural, and was already used in similar form by Diaconis and Holmes~\cite{MR1665632} to establish the existence of a Hamilton path in the flip graph.

\begin{proof}
We use induction on $m$.
For the induction base $m=2$, the flip graph is a triangle, which is Hamilton-connected.
For the induction step, let $m\geq 3$, and let $M$ and~$M'$ be two distinct perfect matchings of~$K_n$.
We label the vertices $1,\ldots,n$ as follows:
Consider two consecutive edges in the symmetric difference~$M\triangle M'$ and label their vertices $2,1,n$ such that $\{1,2\}\in M$ and $\{1,n\}\in M'$.
The remaining vertex labels~$3,\ldots,n-1$ are assigned arbitrarily.

We partition the set of all perfect matchings according to the edge incident with vertex~1.
Specifically, for $i=2,\ldots,n$, we let $\cM_i$ be the set of matchings that contain the fixed edge~$\{1,i\}$; see Figure~\ref{fig:hconn}.
Note that $M\in\cM_2$ and $M'\in\cM_n$.
We choose pairs of matchings $A_i,B_i\in\cM_i$ for $i=2,\ldots,n$ satisfying the following conditions:
\begin{itemize}[topsep=1mm,leftmargin=4mm]
\item $A_2=M$ and $B_n=M'$;
\item $A_i\neq B_i$ for $i=2,\ldots,n$;
\item $A_{i+1}$ differs from $B_i$ by a 2-flip for all $i=2,\ldots,n-1$.
\end{itemize}
For any matching $B_i\in\cM_i$, there is a suitable 2-flip to reach some matching in $\cM_{i+1}$.
Thus, when $A_{n-1}$ is chosen, there are at least two choices for~$B_{n-1}$, and not both of their neighbors in~$\cM_n$ can be equal to the prescribed last matching $B_n=M'$, so one of the neighbors can become~$A_n$.

By induction, for $i=2,\ldots,n$ there is a path~$P_i$ in the flip graph from $A_i$ to $B_i$ through all matchings in~$\cM_i$.
The concatenation $P_2,P_3,\ldots,P_n$ is the desired Hamilton path in the entire flip graph.
\end{proof}

\subsection{Almost-perfect matchings in complete graphs}
\label{sec:CO1-hconn}

An immediate consequence of Theorem~\ref{thm:CE2-hconn} and the aforementioned bijection between perfect matchings of~$K_n$ and almost-perfect matchings of~$K_{n-1}$ is that the flip graph of almost-perfect matchings under 2-flips and 1-flips is Hamilton-connected.

The following theorem strengthens this result and eliminates 2-flips, i.e., we obtain Hamilton-connectedness under 1-flips, which yields an affirmative answer to Question~\ref{quest:workshop} raised by Aichholzer, Dorfer, Rieck and Verciani~\cite{booklet-problem}.

\begin{theorem}
\label{thm:CO1-hconn}
Let $m\geq 1$ and $n\coloneq 2m+1$.
The flip graph of almost-perfect matchings of~$K_n$ under 1-flips is Hamilton-connected.
\end{theorem}

An efficient C++ implementation of the construction in the following proof can be found at~\cite{cos_match}.

In preparation for the proof, we need the following definitions and lemma.
For $n\geq 2$, we define $P_n \coloneq \{(i,j) \mid i,j\in [n] \text{ and } i\neq j\}$ as the set of ordered pairs of distinct entries.
We define a graph~$G_n$ with vertex set~$P_n$ by adding the edges $\{(i,j),(i,k)\}$ and $\{(i,j),(k,j)\}$ for all $i\in[n]$, $j\in[n]\setminus\{i\}$, $k\in[n]\setminus\{i,j\}$, referred to as \defi{grid edges}, and the edges $\{(i,i+1),(i+1,i)\}$ for all $i\in[n]$, referred to as \defi{diagonal edges}; see Figure~\ref{fig:G5}.
All these operations are understood modulo~$n$, with $1,\ldots,n$ as representatives of the equivalence classes.
In particular, the graph~$G_n$ contains the diagonal edge~$\{(1,n),(n,1)\}$.
The graph~$G_n$ captures the operations of changing a single entry in each pair $(i,j)\in P_n$ (grid edges), and of swapping the order of the two entries of the pairs~$(i,i+1)$ for $i\in[n]$ (diagonal edges).

\begin{figure}[h!]
\includegraphics{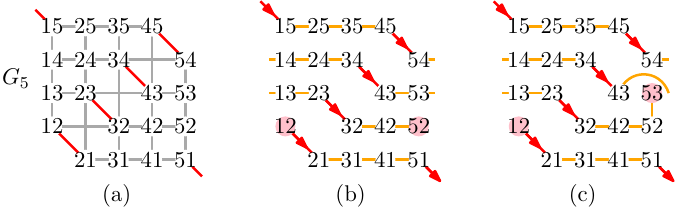}
\caption{(a) Illustration of the graph~$G_5$.
Horizontal and vertical lines indicate grid edges between any two vertices with the same ordinate or abscissa, respectively, whereas red edges are the diagonal edges.
(b)+(c) Illustration of the two Hamilton paths mentioned in Lemma~\ref{lem:pairs}.}
\label{fig:G5}
\end{figure}

\begin{lemma}
\label{lem:pairs}
For any $n\geq 5$, the graph~$G_n$ has a Hamilton path that starts at~$(1,2)$, ends at~$(n,2)$, and contains all diagonal edges, visiting any two pairs $(i,i+1),(i+1,i)$ for $i\in[n]$ in that order.
The same statement holds for the ending pair~$(n,3)$ instead of~$(n,2)$.
\end{lemma}

\begin{proof}
A path~$P$ from~$(1,2)$ to~$(n,2)$ satisfying those constraints is obtained from the rule $(i,j) \rightarrow(i+1,j)$ if $j\neq i+1$ and $(i,i+1) \rightarrow (i+1,i)$ otherwise (modulo $n$); see Figure~\ref{fig:G5}~(b).
A path~$P'$ from~$(1,2)$ to~$(n,3)$ is obtained from~$P$ by replacing the two edges~$\{(n-1,3),(n,3)\}$ and $\{(n,3),(1,3)\}$ by the shortcut edge $\{(n-1,3),(1,3)\}$, and by adding the edge $\{(n,2),(n,3)\}$ at the end; see Figure~\ref{fig:G5}~(c).
\end{proof}

\begin{figure}[b!]
\includegraphics[page=1,scale=0.8]{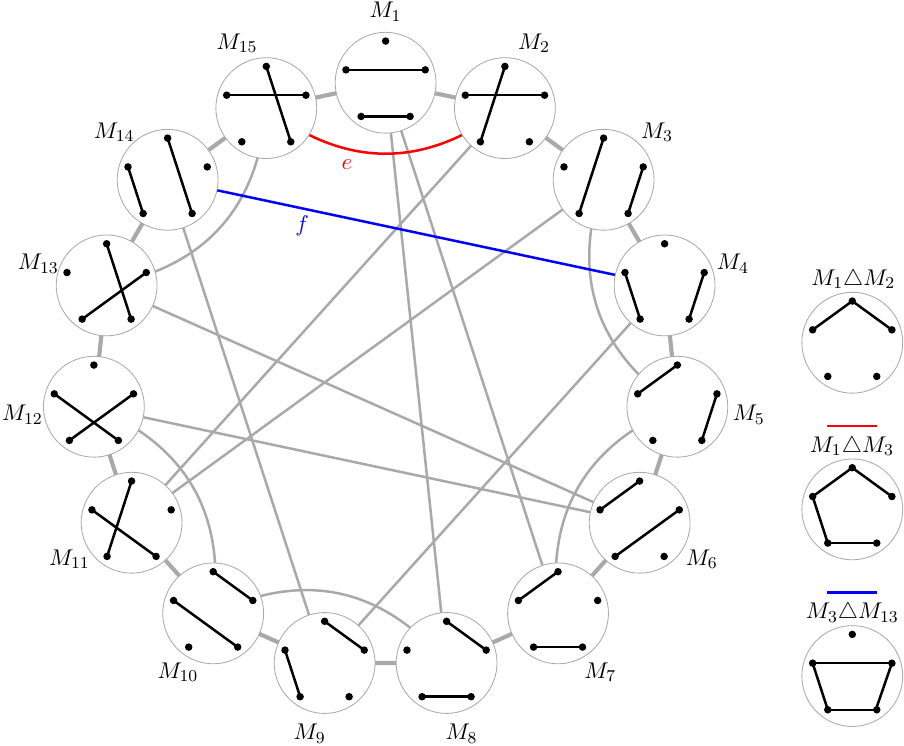}
\caption{The flip graph of almost-perfect matchings of~$K_5$ under 1-flips, with a Hamilton cycle on the perimeter.
By small modifications one obtains Hamilton paths between the pairs of matchings shown at the right.
The graphs obtained by taking the symmetric difference are drawn with the 5~vertices permuted suitably, to show the three different isomorphism classes more clearly (2-path, 4-path, or 4-cycle).}
\label{fig:K5-1flips}
\end{figure}

\begin{figure}[t!]
\makebox[0cm]{ 
\includegraphics[page=3,scale=0.6]{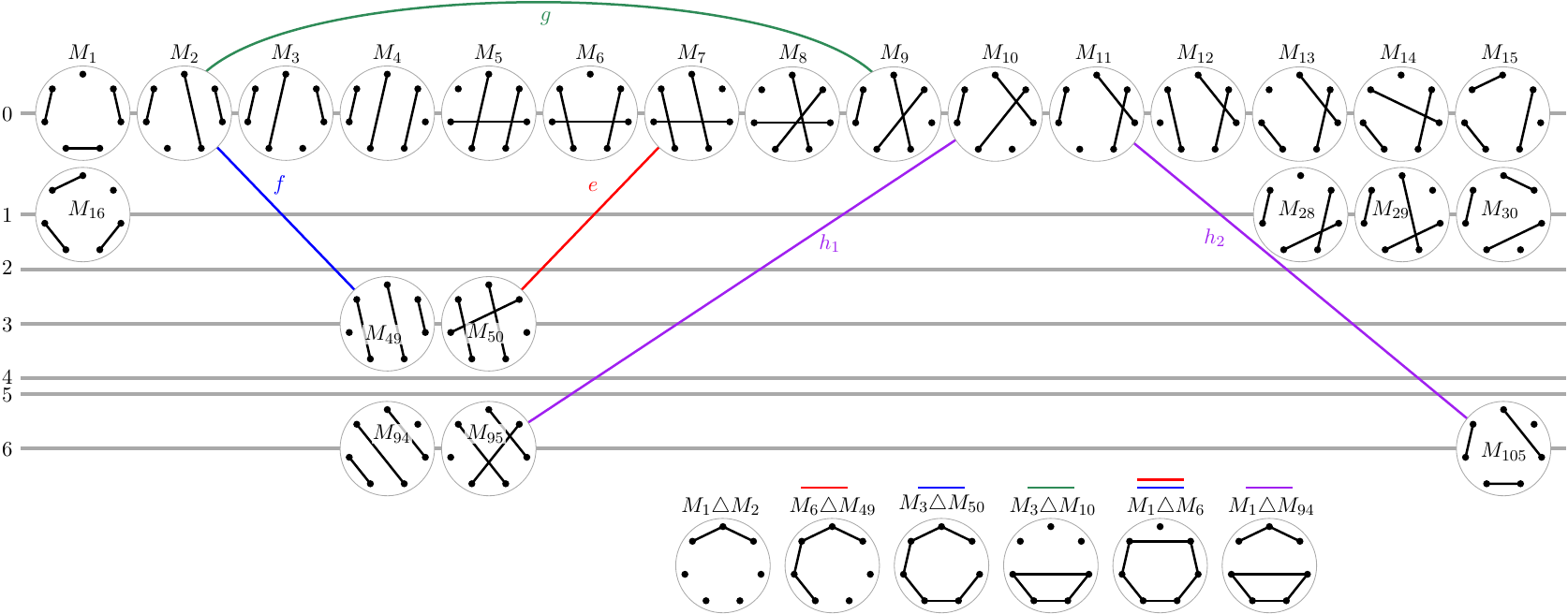}
}
\caption{A symmetric Hamilton cycle in the flip graph of almost-perfect matchings of~$K_7$ under 1-flips.
The matchings in the $i$th row, $i=1,\ldots,6$, are obtained by rotating those in the top row (=row~0) by $2\pi/7\cdot i$ in clockwise direction (not all rotated matchings are shown).
By small modifications one obtains Hamilton paths between the pairs of matchings shown at the bottom.
The graphs obtained by taking the symmetric difference are drawn with the 7~vertices permuted suitably, to show the six different isomorphism classes more clearly (2-path, 4-path, 6-path, 4-cycle, 6-cycle, 2-path plus 4-cycle).}
\label{fig:K7-1flips}
\end{figure}

\begin{proof}[Proof of Theorem~\ref{thm:CO1-hconn}]
We denote the set of almost-perfect matchings of~$K_n$, $n=2m+1$, i.e., matchings with $m$ edges, by~$\cM_m$.
To cover the different possibilities of combining a start and target matching, note that the symmetric difference~$M\triangle M'$ of the edge sets of any two distinct matchings~$M,M'\in\cM_m$ is a non-empty union of cycles and at most one path, all of even length.
Since permutating the vertices of~$K_n$ is obviously a symmetry, all the following arguments only distinguish the isomorphism type of~$M\triangle M'$.

We prove the statement by induction on~$m$.
We first settle the base cases $m\in\{1,2,3\}$.

For $m=1$, the flip graph is a triangle, so the statement holds. 

For $m=2$, the flip graph is shown in Figure~\ref{fig:K5-1flips}, with the Hamilton cycle~$C\coloneq (M_1,\ldots,M_{15})$ on the perimeter.
In the following, we will need the edges $e\coloneq \{M_2,M_{15}\}$ and~$f\coloneq \{M_4,M_{14}\}$.
Given two distinct matchings~$M,M'\in\cM_2$, the symmetric difference~$M\triangle M'$ is either a 2-path, a 4-path, or a 4-cycle.
Hamilton paths covering each of these three cases are given as follows (the symmetric difference between start and target matching is shown in the right part of the figure, with a suitable permutation of the vertices):
\begin{itemize}[leftmargin=4mm,topsep=1mm]
\item from~$M_1$ to $M_2$: $C\setminus \{\{M_1,M_2\}\}$;
\item from~$M_1$ to~$M_3$: $(C\setminus \{\{M_1,M_{15}\},\{M_2,M_3\}\})\cup \{e\}$;
\item from~$M_3$ to~$M_{13}$: $(C\setminus\{\{M_3,M_4\},\{M_{13},M_{14}\}\})\cup \{f\}$.
\end{itemize}

For $m=3$, the $|\cM|=105$ matchings split into 15 equivalence classes under rotations, each of size~7.
A path visiting one representative $M_1,\ldots,M_{15}$ of each class is shown in the top row of Figure~\ref{fig:K7-1flips}.
Applying one more flip visits~$M_{16}$, a rotated copy of~$M_1$, i.e., it leads back to the first equivalence class.
Therefore, repeatedly applying the same (rotated) flips yields a Hamilton cycle~$C\coloneq (M_1,\ldots,M_{105})$ in the flip graph.
In the following we will need the edges $e\coloneq \{M_7,M_{50}\}$, $f\coloneq \{M_2,M_{49}\}$, $g\coloneq \{M_2,M_9\}$, $h_1\coloneq \{M_{10},M_{95}\}$ and $h_2\coloneq \{M_{11},M_{105}\}$.
Given two distinct matchings~$M,M'\in\cM_3$, the symmetric difference~$M\triangle M'$ is either a 2-path, a 4-path, a 6-path, a 4-cycle, a 6-cycle, or a 2-path plus a 4-cycle.
Hamilton paths covering each of these six cases are given as follows (the symmetric difference between start and target matching is shown at the bottom of the figure, with the vertices suitably permuted):
\begin{itemize}[leftmargin=4mm,topsep=1mm]
\item from~$M_1$ to~$M_2$: $C\setminus \{\{M_1,M_2)\}\}$.
\item from~$M_6$ to~$M_{49}$: $(C\setminus \{\{M_6,M_7\},\{M_{49},M_{50}\}\})\cup\{e\}$.
\item from~$M_3$ to~$M_{50}$: $(C\setminus \{\{M_2,M_3\},\{M_{49},M_{50}\}\})\cup\{f\}$.
\item from~$M_3$ to~$M_{10}$: $(C\setminus \{\{M_2,M_3\},\{M_9,M_{10}\}\}\}\cup\{g\}$.
\item from~$M_1$ to~$M_6$: $(C\setminus \{\{M_1,M_2\},\{M_6,M_7\},\{M_{49},M_{50}\}\})\cup\{e,f\}$.
\item from~$M_1$ to~$M_{94}$: $(C\setminus \{\{M_1,M_{105}\},\{M_{94},M_{95}\},\{M_{10},M_{11}\}\})\cup\{h_1,h_2\}$.
\end{itemize}

For the induction step, let $m\geq 4$, and note that the induction step goes back by 2~steps from~$m$ to~$m-2$.
Let $M,M'\in\cM_m$ be two distinct matchings.

Case~(a): We first consider the case that~$M$ and~$M'$ have at least one edge in common.
We label the vertices $1,\ldots,n$ as follows; see Figure~\ref{fig:tick1}~(a):
The end vertices of the edge that is common to~$M$ and~$M'$ receive the labels~2 and~$n$.
Furthermore, we choose two consecutive edges from the symmetric difference~$M\triangle M'$ and label the three vertices $1,n-1,n-2$ such that $\{n-1,1\}\in M$ and~$\{n-1,n-2\}\in M'$.
The remaining vertex labels~$3,\ldots,n-3$ are assigned arbitrarily.

We partition the set~$\cM_m$ into several sets of matchings as follows; see Figure~\ref{fig:tick2}:
\begin{itemize}[leftmargin=4mm,topsep=1mm]
\item $\cX_{i,j}$ for $i,j\in[n-2]$, $i\neq j$: matchings in which both edges~$\{n-1,i\}$ and~$\{n,j\}$ are present.
\item $\cY_i$ for $i\in[n-2]$: matchings in which vertex~$n$ is unmatched and the edge~$\{n-1,i\}$ is present.
\item $\cY_i'$ for $i\in[n-2]$: matchings in which vertex~$n-1$ is unmatched and the edge~$\{n,i\}$ is present.
\item $\cZ_i$ for $i\in[n-2]$: matchings in which vertex~$i$ is unmatched and the edge~$\{n-1,n\}$ is present.
\end{itemize}
Note that $M\in\cX_{1,2}$ and $M'\in\cX_{n-2,2}$, and furthermore $|\cX_{i,j}|=|\cY_i|=|\cY_i'|=|\cZ_i|=|\cM_{m-2}|$ (and these cardinalities are independent of $i$ and~$j$).

\begin{figure}
\includegraphics[page=1]{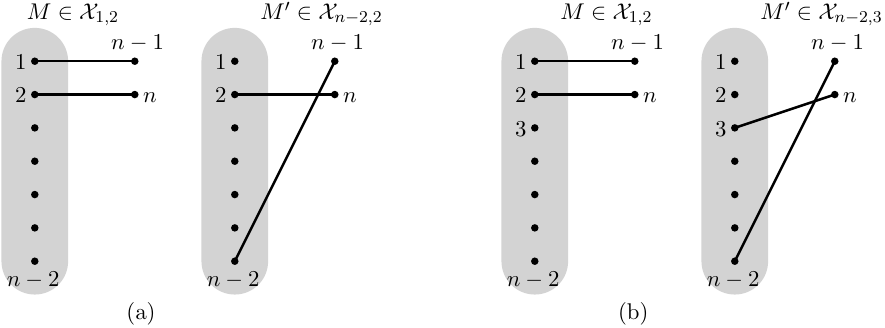}
\caption{Illustration of the two cases in the proof of Theorem~\ref{thm:CO1-hconn} and the corresponding labeling of vertices in the start and target matchings.}
\label{fig:tick1}
\end{figure}

\begin{figure}[b!]
\makebox[0cm]{ 
\includegraphics[page=2]{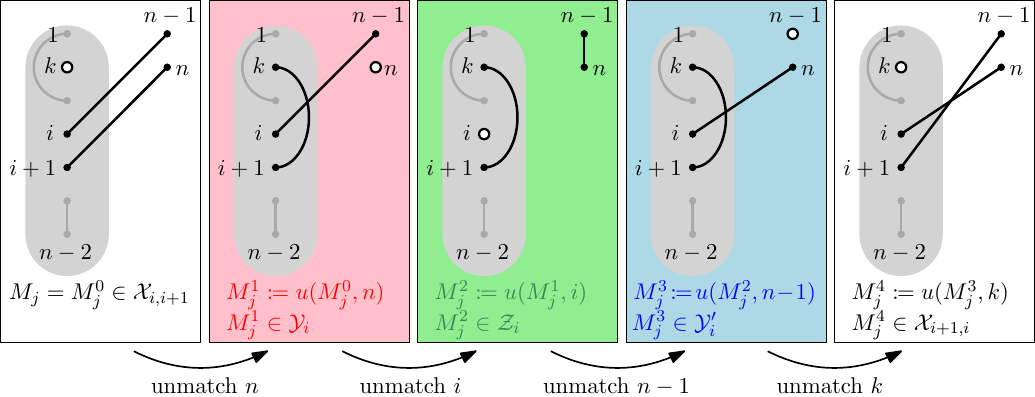}
}
\caption{Modified flip sequence to include matchings in the sets~$\cY_i\cup \cZ_i\cup \cY_i'$.}
\label{fig:tick2}
\end{figure}

We apply Lemma~\ref{lem:pairs} to obtain a listing~$L$ of all pairs $(i,j)$, $i,j\in[n-2]$, $i\neq j$, that starts at $(i,j)=(1,2)$ and ends at $(n-2,2)$ such that all pairs $(i,i+1),(i+1,i)$ for $i\in[n-2]$ appear consecutively.
The ordering~$L$ determines the order in which we visit the matchings of~$\cM_m$, namely, it determines the neighbors to which the vertices $n-1$ and~$n$ are matched.
Specifically, for every pair~$(i,j)$ in~$L$ with $j\neq i\pm 1$, we visit all matchings in the set~$\cX_{i,j}$, and for every two consecutive pairs~$(i,i+1),(i+1,i)$ in~$L$, we visit all matchings in~$\cX_{i,i+1}\cup \cY_i\cup \cZ_i\cup \cY_i'\cup \cX_{i+1,i}$.

We now describe these two steps in more detail; see Figures~\ref{fig:tick-tick} and~\ref{fig:M91} for illustrations.
We first define~$A_{1,2}\coloneq M$ and $B_{n-2,2}\coloneq M'$.

\begin{figure}[t!]
\includegraphics{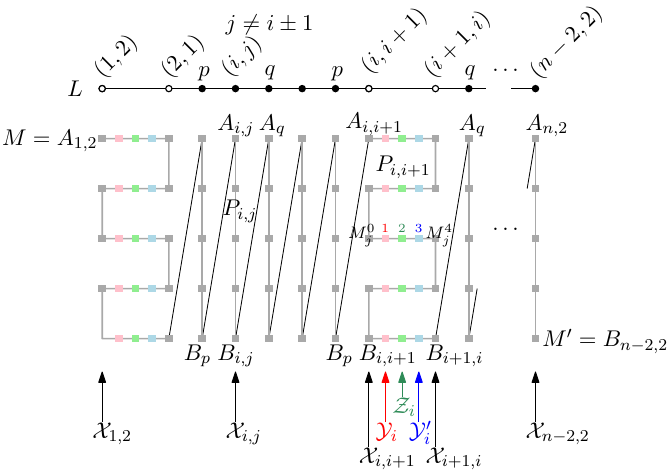}
\caption{Structure of the resulting Hamilton path.}
\label{fig:tick-tick}
\end{figure}

For every pair~$(i,j)$ in~$L$ with $j\neq i\pm 1$, we select two matchings~$A_{i,j},B_{i,j}\in\cX_{i,j}$, as follows:
Let $p$ and~$q$ be the pairs directly before and after~$(i,j)$ in~$L$, respectively (if they exist).
We let $A_{i,j}$ be the matching obtained from~$B_p$ by the uniquely determined 1-flip that matches $n-1$ and~$n$ to $i$ and~$j$, respectively (one of them already has the correct neighbor).
Unless $(i,j)=(n-2,2)$ is the last pair, we let $B_{i,j}$ be an arbitrary matching distinct from~$A_{i,j}$, such that if $q=(k,j)$ or $q=(i,k)$, then the vertex~$k$ is unmatched in~$B_{i,j}$.
This distinct choice is possible even if the unmatched vertex is the same in $A_{i,j}$ and~$B_{i,j}$, by the assumption that $m\geq 4$, as two matching edges can be chosen freely.

\begin{figure}[b!]
\makebox[0cm]{ 
\includegraphics[page=1,scale=1.4]{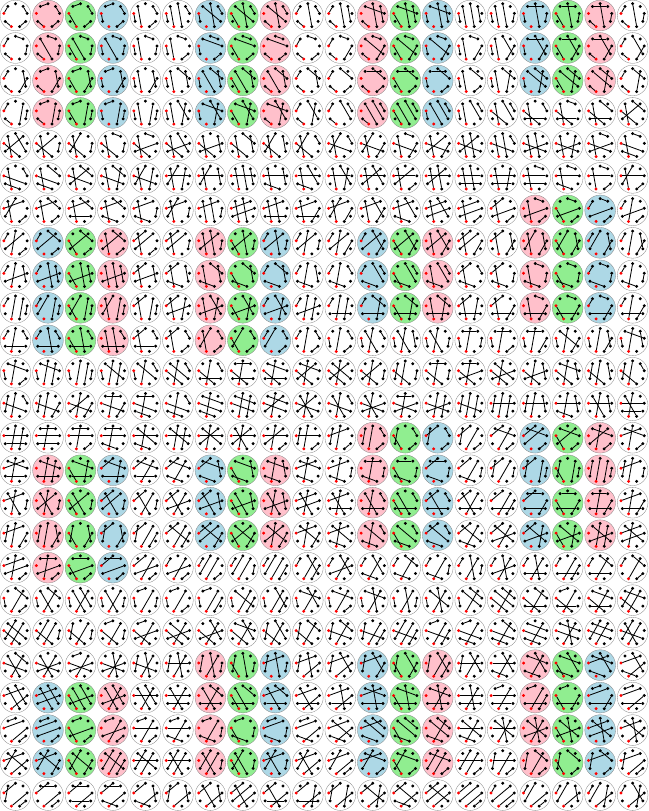}
}
\caption{Example of the construction for $m=4$ (read line by line from top to bottom); matchings from the sets $\cY_i,\cZ_,\cY_i'$ are highlighted in red, green, blue, respectively, as in Figures~\ref{fig:tick2} and~\ref{fig:tick-tick}, and the two vertices~$n$ and~$n-1$ separated in the induction step are drawn in red (figure continues on next page).}
\label{fig:M91}
\end{figure}

\addtocounter{figure}{-1}
\begin{figure}[t!]
\makebox[0cm]{ 
\includegraphics[page=2,scale=1.4]{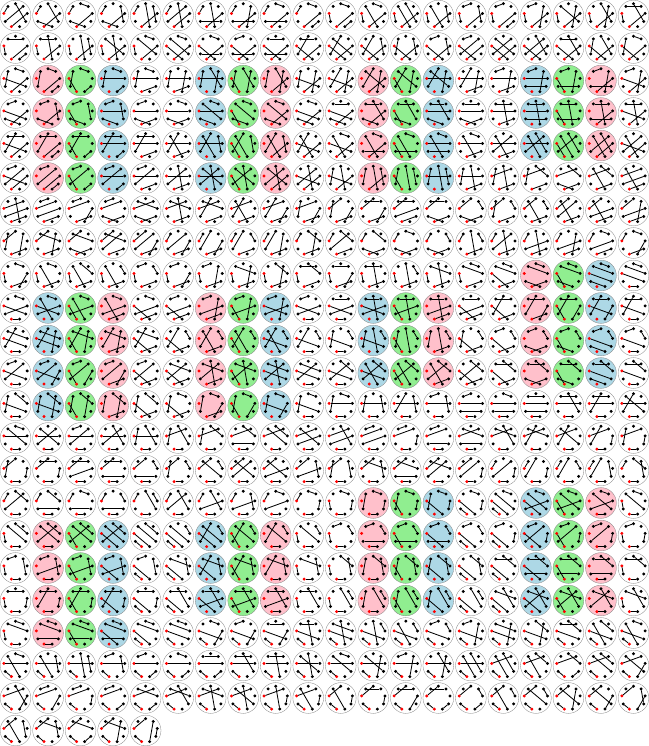}
}
\caption{(figure continued from previous page).}
\label{fig:M92}
\end{figure}

For every two consecutive pairs~$(i,i+1),(i+1,i)$ in~$L$, we select~$A_{i,i+1}\in\cX_{i,i+1}$ and $B_{i+1,i}\in\cX_{i+1,i}$, as follows:
Let $p$ and~$q$ be the pairs directly before~$(i,i+1)$ and after~$(i+1,i)$ in~$L$, respectively (if they exist).
Unless $(i,i+1)=(1,2)$ is the first pair, we let $A_{i,i+1}$ be the matching obtained from~$B_p$ by the uniquely determined 1-flip that matches $n-1$ and~$n$ to $i$ and~$i+1$, respectively (one of them already has the correct neighbor).
We then let $B_{i,i+1}\in\cX_{i,i+1}$ be an arbitrary matching distinct from~$A_{i,i+1}$, such that if $q=(k,i)$ or $q=(i+1,k)$, then the vertex~$k$ is unmatched in~$B_{i,i+1}$.
This is possible even if the unmatched vertex is the same in $A_{i,i+1}$ and~$B_{i,i+1}$, by the assumption that $m\geq 4$, as two matching edges can be chosen freely.
We then let $B_{i+1,i}$ be the matching obtained from~$B_{i,i+1}$ by exchanging the neighbors of the vertices~$n-1$ and~$n$ (to be $i+1$ and~$i$ instead of $i$ and~$i+1$). 

We note that for the last two pairs $p,(n-2,2)$ in~$L$, as $m\geq 4$, there are at least two choices for~$B_p$ in each of the two cases before, leading to two different possibilities for $A_{n-2,2}$, at least one of which is different from the prescribed last matching~$B_{n-2,2}=M'$.

By induction, for every pair~$(i,j)$ in~$L$ with $j\neq i\pm 1$, there is a path~$P_{i,j}$ in the flip graph from~$A_{i,j}$ to~$B_{i,j}$ through all matchings in~$\cX_{i,j}$.
Indeed, since the edges $\{n-1,i\}$ and $\{n,j\}$ are fixed, the set~$\cX_{i,j}$ and 1-flips between them is isomorphic to the flip graph on~$\cM_{m-2}$.

For every two consecutive pairs~$(i,i+1),(i+1,i)$ in~$L$, by induction there is a path~$Q=(M_1,\ldots,M_\ell)$ in the flip graph from~$A_{i,i+1}$ to~$B_{i,i+1}$ through all matchings in~$\cX_{i,i+1}$.
We modify this path so that it also includes all matchings in $\cY_i\cup \cZ_i\cup \cY_i'\cup \cX_{i+1,i}$ and ends at~$B_{i+1,i}$.
Indeed, starting with any matching~$M_j^0\coloneq M_j\in\cX_{i,i+1}$, $j\in[\ell]$, with unmatched vertex~$k\in[n-2]\setminus\{i,i+1\}$, the flip sequence of unmatching first~$n$, then~$i$, then~$n-1$, then~$k$, creates four additional matchings $M_j^1\coloneq u(M_j^0,n),M_j^2\coloneq u(M_j^1,i),M_j^3\coloneq u(M_j^2,n-1),M_j^4\coloneq u(M_j^3,k)$.
By definition we have that $M_j^1\in\cY_i$, $M_j^2\in\cZ_i$, $M_j^3\in\cY_i'$, $M_j^4\in\cX_{i+1,i}$, i.e., $M_j^4$ differs from~$M_j^0$ only in exchanging the neighbors of~$n-1$ and~$n$; see Figure~\ref{fig:tick2}.
We define $M_j^*\coloneq (M_j^0,\ldots,M_j^4)$ for all $j\in[\ell]$.
Using that $\ell=|\cM_{m-2}|$ is always odd, we see that the sequence $P_{i,i+1}\coloneq (M_1^*,\rev(M_2^*),M_3^*,\rev(M_4^*),\ldots,\rev(M_{\ell-1}^*),M_\ell^*)$ is a path through all matchings in~$\cX_{i,i+1}\cup \cY_i\cup \cZ_i\cup \cY_i'\cup \cX_{i+1,i}$ that starts at~$A_{i,i+1}$ and ends at~$B_{i+1,i}$.

The concatenation of the paths~$P_{i,j}$ for all pairs $(i,j)$ with $j\neq i\pm 1$ and of the paths $P_{i,i+1}$ for the consecutive pairs~$(i,i+1),(i+1,i)$, in the order in which they appear along~$L$, is the desired Hamilton path in the entire flip graph.

Case~(b): It remains to consider the case that $M$ and~$M'$ have no edges in common.
We label the vertices $1,\ldots,n$ as follows; see Figure~\ref{fig:tick1}~(b):
We choose two vertex-disjoint pairs of consecutive edges from the symmetric difference~$M\triangle M'$ and label the three vertices $1,n-1,n-2$ and $2,n,3$ such that $\{n-1,1\},\{n,2\}\in M$ and~$\{n-1,n-2\},\{n,3\}\in M'$.
The remaining vertex labels~$4,\ldots,n-3$ are assigned arbitrarily.
From this point the proof proceeds similarly to before, using instead a listing~$L$ of pairs that starts at~$(1,2)$ and ends at~$(n-2,3)$, as provided by Lemma~\ref{lem:pairs}.
\end{proof}

We generalize Theorem~\ref{thm:CO1-hconn} as follows.

\begin{theorem}
\label{thm:COd-hconn}
Let $m\geq 1$, $d\geq 1$ and $n\coloneq 2m+d$.
The flip graph of $m$-edge matchings of~$K_n$ under 1-flips is Hamilton-connected.
\end{theorem}

It turns out that the case~$d=1$ is the hardest one, and is settled by Theorem~\ref{thm:CO1-hconn}, whereas all other cases can be settled by a straightforward induction.

\begin{proof}
We also extend the statement to the trivial case~$m=0$, in which case there is only one matching (the empty set), i.e., the flip graph is a single vertex.
On the other hand, for any $m\geq 1$ and $d=1$ the statement is established by Theorem~\ref{thm:CO1-hconn}.
To prove the remaining cases, we argue by induction on~$m$ and~$d$.

Specifically, let $m\geq 1$, $d\geq 2$ and $n\coloneq 2m+d$, and let $M$ and~$M'$ be two distinct $m$-edge matchings of~$K_n$.
We label the vertices~$1,\ldots,n$ as follows:
If $M\triangle M'$ contains a path, we label the vertices of one of its terminal edges~$1,2$ such that vertex~1 is unmatched in~$M$ and $\{1,2\}\in M'$.
If $M\triangle M'$ contains no paths, but only cycles, then we consider two consecutive edges on one of them and label their vertices $2,1,3$ such that $\{1,2\}\in M$ and $\{1,3\}\in M'$.
The remaining vertex labels~$3,\ldots,n$ or $4,\ldots,n$, respectively, are assigned arbitrarily.

We partition the set of all $m$-edge matchings of~$K_n$ according to whether the vertex~1 is matched, and if so, according to the edge incident with vertex~1.
Specifically, we let $\cM_1$ be the set of matchings in which vertex~1 is unmatched, and for $i=2,\ldots,n$ we let $\cM_i$ be the set of matchings that contain the fixed edge~$\{1,i\}$.
We choose a total ordering~$\cN=(\cN_1,\ldots,\cN_n)$ of the sets $\cM_1,\ldots,\cM_n$ such that $M\in\cN_1$ and~$M'\in\cN_n$.
We then choose pairs of matchings $A_i,B_i\in\cN_i$ for $i=1,\ldots,n$ satisfying the following conditions:
\begin{itemize}[topsep=1mm,leftmargin=4mm]
\item $A_1=M$ and $B_n=M'$;
\item if $|\cN_i|>1$ then $A_i\neq B_i$ for $i=1,\ldots,n$;
\item $A_{i+1}$ differs from $B_i$ by a 1-flip for all $i=1,\ldots,n-1$.
\end{itemize}
If $m=1$, then if $\cN_i=\cM_j$ for some $j\in[2,n]$ we have $|\cN_i|=1$, so~$A_i$ and~$B_i$ are uniquely determined, and if $\cN_i=\cM_1$ and $\cN_{i+1}=\cM_j$ for some~$j\in[2,n]$, then we choose $B_i$ as a matching in which vertex~$j$ is matched.
If $m\geq 2$, on the other hand, then each set~$\cN_i$, $i=1,\ldots,n-1$, contains at least three matchings for which a 1-flip leads to a matching in~$\cN_{i+1}$.
Specifically, if $\cN_i=\cM_j$ for some $j\in[2,n]$ and $\cN_{i+1}=\cM_k$ for some $k\in[2,n]$, then these are matchings in which vertex~$k$ is unmatched.
If $\cN_i=\cM_j$ for some $j\in[2,n]$ and~$\cN_{i+1}=\cM_1$, then all matchings in~$\cN_i$ have this property.
Lastly, if $\cN_i=\cM_1$ and~$\cN_{i+1}=\cM_j$ for some $j\in[2,n]$, then these are matchings in which vertex~$j$ is matched.
Thus, when $A_{n-1}$ is chosen, there are at least two choices for~$B_{n-1}$, and not both of their neighbors in~$\cN_n$ can be equal to the prescribed last matching $B_n=M'$, so one of the neighbors can become~$A_n$.

By induction, for $i=1,\ldots,n$ there is a path~$P_i$ in the flip graph from~$A_i$ to~$B_i$ through all matchings in~$\cN_i$.
Indeed, if $\cN_i=\cM_1$, then the induction step fixes the unmatched vertex~1 (so $m$ is unchanged and $d$ is decremented by~1), whereas if $\cN_i=\cM_j$ for some $j\in[2,n]$, then the induction step fixes the edge~$\{1,i\}$ (so $m$ is decremented by~1 and $d$ is unchanged).
The concatenation $P_1,P_2,\ldots,P_n$ is the desired Hamilton path in the entire flip graph.
\end{proof}

\subsection{Perfect matchings in complete bipartite graphs}

Perfect matchings of~$K_{n,n}$ are in one-to-one correspondence with permutations of the set~$[n]$; see Figure~\ref{fig:bip}~(a).
Indeed, if we label the vertices in both partition classes by $1,\ldots,n$, then every perfect matching is a collection of edges $(i,\pi(i))$, where $i$ is a vertex from the first class and $\pi(i)$ is its matching partner in the second class, which we can encode compactly as the permutation $(\pi(1),\ldots,\pi(n))$ in one-line notation.
Consequently, those matchings are counted by the factorial numbers~$n!$.
Under this bijection, 2-flips in~$K_{n,n}$ corresponds to transpositions in the permutations.
Specifically, a 2-flip involving the edges $(i,\pi(i))$ and $(j,\pi(j))$ corresponds to the transposition $\pi(i)\leftrightarrow \pi(j)$.

\begin{theorem}[Tchuente \cite{MR683982}]
\label{thm:tchuente}
For any~$n\geq 4$ and any transposition tree on~$[n]$, the flip graph of permutations of~$[n]$ under transpositions from this tree is Hamilton-laceable.
\end{theorem}

The transposition tree is a graph with vertex set~$[n]$, and an edge~$\{i,j\}$ present in the graph indicates that the transposition $\pi(i)\leftrightarrow \pi(j)$ is valid.

\begin{figure}[b!]
\includegraphics{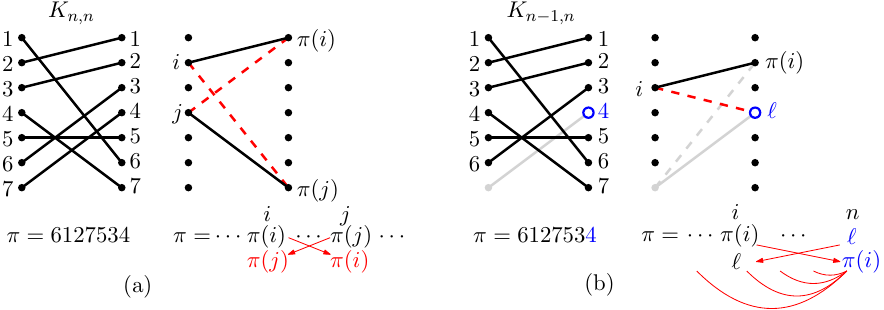}
\caption{Correspondence between (a) perfect matchings of~$K_{n,n}$ and permutations, and between 2-flips and transpositions; (b) almost-perfect matchings of~$K_{n-1,n}$ and permutations, and between 1-flips and star transpositions.}
\label{fig:bip}
\end{figure}

An immediate consequence of Theorem~\ref{thm:tchuente} and the bijection between perfect matchings of~$K_{n,n}$ and permutations on~$[n]$ is the following. 

\begin{theorem}
\label{thm:CEB-lace}
For any $n\geq 2$, the flip graph of perfect matchings of~$K_{n,n}$ under 2-flips is Hamilton-laceable.
\end{theorem}

\begin{proof}
For $n=2$ the flip graph is a single edge, and for $n=3$ the flip graph is $K_{3,3}$, both of which are Hamilton-laceable.
For $n\geq 4$ the statement follows from Theorem~\ref{thm:tchuente}.
\end{proof}

As in Tchuente's result, one may even restrict the allowed 2-flips further, so that they correspond to a tree of transpositions, and still obtain Hamilton-laceability.
Note however, that as any transposition changes the parity of a permutation, the start and end vertex of a Hamilton path may never have the same parity, i.e., we cannot hope to strengthen the theorem to Hamilton-connectedness.

\subsection{Almost-perfect matchings in complete bipartite graphs}

We observe that the set of perfect matchings of~$K_{n,n}$ is in one-to-one correspondence with the set of almost-perfect matchings of~$K_{n-1,n}$; see Figure~\ref{fig:bip}~(b).
Indeed, the bijection is given by removing one vertex from the first partition class, the last one, say.
We can encode an almost-perfect matching of $K_{n-1,n}$ by a permutation of length~$n$, where we first write the matching partners of the vertices in the first partition class, and the unmatched vertex last.
Note that a 1-flip in the almost-perfect matching corresponds to a \defi{star transposition} in the permutation, i.e., a transposition of the last entry with some earlier entry.
Specifically, the unmatched vertex (listed last) becomes matched and instead a matched vertex becomes unmatched.
Thus, applying Theorem~\ref{thm:tchuente} with a star graph centered at the last index as transposition tree gives the following result.

\begin{theorem}
\label{thm:COB-lace}
For any $n\geq 3$, the flip graph of almost-perfect matchings of $K_{n,n+1}$ under 1-flips is Hamilton-laceable.
\end{theorem}

For $n=2$ the flip graph is a 6-cycle, so Hamilton-laceability does not hold.

In analogy to Theorem~\ref{thm:COd-hconn}, we extend Theorem~\ref{thm:COB-lace} as follows.

\begin{theorem}
\label{thm:COB-conn}
For any $n\geq 1$ and $d\geq 2$, the flip graph of $n$-edge matchings of~$K_{n,n+d}$ under 1-flips is Hamilton-connected.
\end{theorem}

We can think of such a perfect matching as pair~$(\pi,D)$, where $\pi=(\pi(1),\ldots,\pi(n))$ are the matching partners of the vertices in the first partition class, i.e., $\pi(i)\in[n+d]$ for all $i\in[n]$, and $D=[n+d]\setminus\{\pi(1),\ldots,\pi(n)\}$ is the set of the unmatched vertices in the second partition class.
Each 1-flip exchanges one entry from the sequence~$\pi$ with an element from the set~$D$.
This can be thought of as a star transposition where multiple symbols are set apart (in $D$), generalizing the aforementioned case $d=1$ (in which case $D$ is a singleton and so $(\pi,D)$ can be viewed as a permutation of~$[n+1]$).

\begin{proof}
For $n=1$ and and~$d\geq 2$ the flip graph is complete, so the statement is trivially true.

\begin{figure}[t!]
\includegraphics{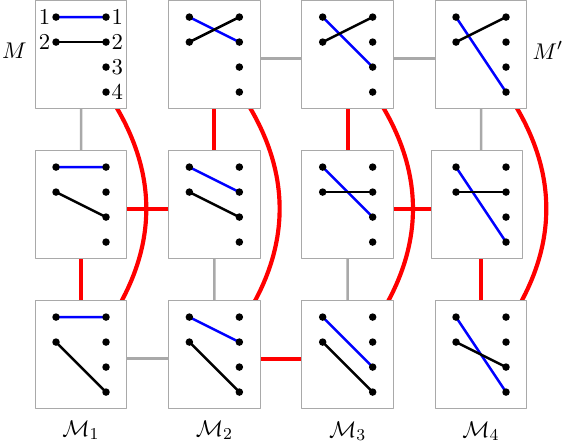}
\caption{Illustration of the proof of Theorem~\ref{thm:COB-conn}.
Edges between sets~$\cM_i$ and~$\cM_j$ for $j>i+1$ are not shown.}
\label{fig:bip-hconn}
\end{figure}

For $n\geq 2$, let $M$ and~$M'$ be two distinct $n$-edge matchings of~$K_{n,n+d}$.
We assume w.l.o.g.\ that $(1,1),(2,2)\in M$ and $(1,n+d)\in M'$.
We partition the set of all $n$-edge matchings of~$K_{n,n+d}$ according to the edge incident with vertex~1 in the first partition class.
Specifically, for $i=1,\ldots,n+d$, we let $\cM_i$ be the set of matchings that contain the fixed edge~$(1,i)$.
For $n=2$ and $d=2$ this partition is shown in Figure~\ref{fig:bip-hconn}, including a Hamilton path from~$M$ to one of three possible matchings~$M'\in\cM_4$.
Solutions for the other two cases can be obtained by straightforward modifications.
For the rest of the proof we assume that $n>2$ or $d>2$.
We choose pairs of matchings $A_i,B_i\in\cM_i$ for $i=1,\ldots,n+d$ satisfying the following conditions:
\begin{itemize}[topsep=1mm,leftmargin=4mm]
\item $A_1=M$ and $B_{n+d}=M'$;
\item $A_i\neq B_i$ for $i=1,\ldots,n+d$;
\item $A_{i+1}$ differs from $B_i$ by a 1-flip for all $i=1,\ldots,n+d-1$.
\end{itemize}
By the assumption $n>2$ or $d>2$, each set~$\cM_i$ contains at least three matchings for which a 1-flip leads to a matching in~$\cM_{i+1}$.
Specifically these matchings in~$\cM_i$ have the vertex~$i+1$ unmatched.
Thus, when $A_{n+d-1}$ is chosen, there are at least two choices for~$B_{n+d-1}$, and not both of their neighbors in~$\cM_{n+d}$ can be equal to the prescribed last matching $B_{n+d}=M'$, so one of the neighbors can become~$A_{n+d}$.

By induction, for $i=1,\ldots,n+d$ there is a path~$P_i$ in the flip graph from~$A_i$ to~$B_i$ through all matchings in~$\cM_i$.
The concatenation~$P_1,P_2,\ldots,P_{n+d}$ is the desired Hamilton path in the entire flip graph.
\end{proof}

\section{Directed matchings in complete graphs}
\label{sec:directed}

\subsection{Directed perfect matchings}

In this section, we consider \emph{directed} perfect matchings of~$K_n$, where $n=2m$ for some integer~$m\geq 1$, i.e., perfect matchings in which every edge has one of two possible orientations.
The total number of such matchings is $2^m(n-1)!!$ (OEIS A001813).

Flips between matchings occur by taking the symmetric difference with a \emph{directed} 4-cycle, of which there are five different types; see Figure~\ref{fig:dir-2flips}.

\begin{figure}[h!]
\includegraphics[page=1]{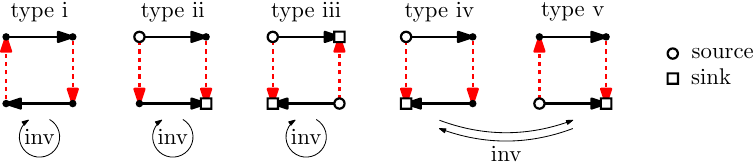}
\caption{The five different types of directed 2-flips.
Solid edges are in the original matching, and dashed edges are in the target matching.
Sources and sinks along the 4-cycles are highlighted, as they distinguish the different types.
The bottom of the figure indicates which flip operations are inverse to each other.}
\label{fig:dir-2flips}
\end{figure}

\begin{figure}[t!]
\includegraphics[page=1]{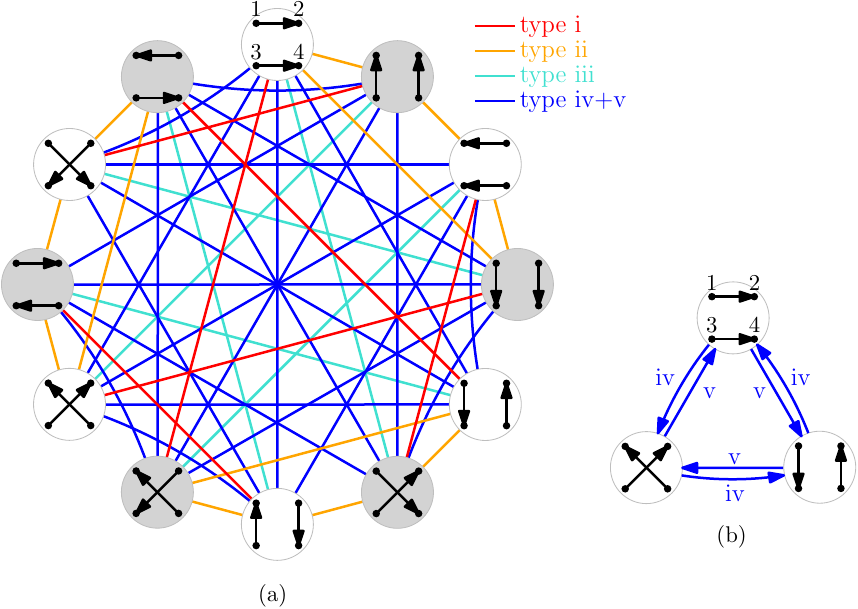}
\caption{(a) Flip graph of directed perfect matchings of~$K_4$, with edges colored according to the different flip types.
For type~iv and~v, the same color is used, and each of these two types corresponds to one of the two orientations of the edge.
Matchings with even and odd parity are non-shaded and shaded, respectively.
(b)~Subgraph of~(a) with type-iv and type-v flips distinguished, showing that two type-iv flips simulate one type-v flip, and vice versa.}
\label{fig:G4d}
\end{figure}

Note that flips of type~i, ii and iii are self-inverse, whereas type~iv is the inverse of type~v and vice versa.
In the former cases, the flip graph is undirected, whereas in the latter cases the flip graph is directed.
This subtlety is irrelevant when it comes to connectivity, because one type-iv flip can be simulated by two type-v flips, and vice versa; see Figure~\ref{fig:G4d}~(b).
Furthermore, in our Hamiltonicity results (Theorem~\ref{thm:CED-hconn}), we consider both flip types~iv and~v together.

Observe that for each flip type (and combinations of them), the corresponding flip graph is always vertex-transitive.
This is because, similarly to the undirected case, any two directed perfect matchings of~$K_n$ are isomorphic under a suitable permutation of the vertices.

\subsubsection{Parity of permutations and directed matchings}

For a permutation~$\pi=(\pi_1,\ldots,\pi_n)$ of~$[n]=[2m]$, we let $M_\pi$ denote the directed perfect matching with the edges $(\pi_1,\pi_2),(\pi_3,\pi_4),\ldots,(\pi_{n-1},\pi_n)$.
Note that exactly $m!$ permutations of~$[n]$ correspond to the same matching~$M_\pi$, namely the ones that differ only in the ordering of the pairs of consecutive entries~$(\pi_{2i-1},\pi_{2i})$, $i\in[m]$.
The parity of all these permutations is the same, because swapping two pairs of consecutive entries $(\pi_{2i-1},\pi_{2i})\leftrightarrow (\pi_{2j-1},\pi_{2j})$ can be seen as a sequence of two transpositions $\pi_{2i-1}\leftrightarrow \pi_{2j-1}$ and $\pi_{2i}\leftrightarrow \pi_{2j}$.
In this way, we can assign a \defi{parity} to each matching~$M$, as the parity of one of the permutations~$\pi$ such that $M=M_\pi$; see Figure~\ref{fig:G4d}~(a).

\subsubsection{Basic properties}

\begin{theorem}
\label{thm:CED-basic}
Let $m\geq 2$ and~$n\coloneq 2m$.
The flip graph of directed perfect matchings of~$K_n$ under 2-flips of exactly one of the five types has the following properties:
\begin{itemize}[topsep=1mm,leftmargin=4mm]
\item type~i: bipartite and connected for $m\geq 4$;
\item type~ii: bipartite and connected for $m\geq 3$;
\item type~iii: bipartite and has exactly $\binom{n}{m}=\binom{2m}{m}$ isomorphic connected components;
\item type~iv or v: non-bipartite and has exactly two isomorphic connected components.
\end{itemize}
\end{theorem}

For $m=2,3$ the flip graph under type-i flips is disconnected, namely, it has six connected components.
For $m=2$ those can be seen in Figure~\ref{fig:G4d}~(a).
For $m=2$ the flip graph under type-ii flips is disconnected, namely, it has three connected components; see Figure~\ref{fig:G4d}~(a).

\begin{proof}
The proof is split into cases according to the different types of flips.

{\bf Type-i flips:}
To argue that the flip graph is bipartite, we consider a permutation~$\pi$ associated to a matching.
A type-i flip in the matching changes two pairs of entries $(\pi_{2i-1},\pi_{2i})$ and~$(\pi_{2j-1},\pi_{2j})$ to $(\pi_{2j},\pi_{2i-1})$ and $(\pi_{2i},\pi_{2j-1})$, which can be seen as a sequence of three transpositions $\pi_{2i-1}\leftrightarrow \pi_{2j}$ followed by $\pi_{2i}\leftrightarrow \pi_{2i-1}$ followed by $\pi_{2j-1}\leftrightarrow \pi_{2i}$.
Consequently, each flip toggles the parity of the matching, so the partition classes of the flip graph are given by the parity of the matchings.

To prove connectivity, Figure~\ref{fig:d2flips-typei} shows a sequence of seven type-i flips on a set of 4~edges that reverses the orientation of exactly one of the edges, and does not change any other edges.

\begin{figure}[h!]
\includegraphics[page=2,scale=0.8]{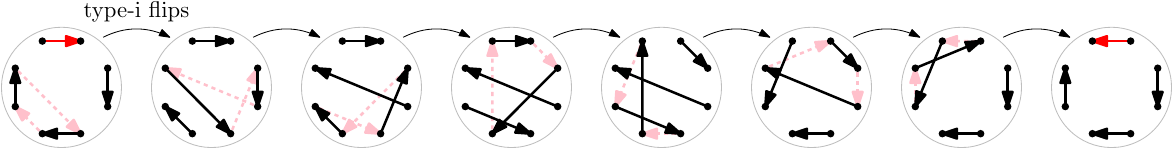}
\caption{Sequence of type-i flips to reverse a single edge.}
\label{fig:d2flips-typei}
\end{figure}
This is enough to argue that any directed perfect matching can be reached from any other via type-i flips.

{\bf Type-ii flips:}
Similarly to type-i flips, each type-ii flip toggles the parity of a permutation~$\pi$ associated to the matching, implying that the flip graph is bipartite, with the partition classes given by the parity of the matchings.
Specifically, a type-ii flip in the matching changes two pairs of entries $(\pi_{2i-1},\pi_{2i})$ and~$(\pi_{2j-1},\pi_{2j})$ to $(\pi_{2i},\pi_{2j-1})$ and $(\pi_{2i},\pi_{2j})$, which can be seen as a single transposition $\pi_{2i}\leftrightarrow \pi_{2j-1}$.

To prove connectivity, Figure~\ref{fig:d2flips-typeii} shows a sequence of five type-ii flips on a set of 3~edges that reverses the orientation of exactly one of the edges, and does not change any other edges.
\begin{figure}[h!]
\includegraphics[page=3,scale=0.8]{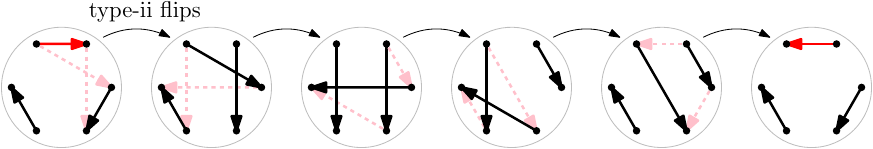}
\caption{Sequence of type-ii flips to reverse a single edge.}
\label{fig:d2flips-typeii}
\end{figure}
This is enough to argue that any directed perfect matching can be reached from any other via type-ii flips.

{\bf Type-iii flips:}
Note that under this type of flip, every vertex retains the property of being incident with an out-edge or in-edge of the matching, respectively.
Thus, the entire flip graph has $\binom{n}{m}$ connected components, one for each choice of which $m$ vertices out of the $n$ vertices have out-edges, and each of these components is isomorphic to the flip graph for perfect matchings of~$K_{m,m}$ discussed in Theorem~\ref{thm:CEB-lace}.

{\bf Type-iv or type-v flips:}
As they are inverse to each other, the two corresponding flip graphs differ only in the orientation of their edges.
Thus, for the rest of the proof it is sufficient to consider only type-iv flips.

As can be see from Figure~\ref{fig:G4d}, the flip graph has triangles, so it is not bipartite.

To argue that the flip graph has at least two connected components, consider a permutation~$\pi$ associated to a matching.
A type-iv flip in the matching changes two pairs of entries $(\pi_{2i-1},\pi_{2i})$ and~$(\pi_{2j-1},\pi_{2j})$ to $(\pi_{2i-1},\pi_{2j})$ and $(\pi_{2i},\pi_{2j-1})$, which can be seen as a sequence of two transpositions $\pi_{2i}\leftrightarrow \pi_{2j}$ followed by $\pi_{2j-1}\leftrightarrow \pi_{2j}$.
Consequently, applying a flip preserves the parity of the matching, and never toggles it.
Thus the flip graph has at least two connected components, induced by all matchings with even or odd parity.

It remains to show that each of these two subgraphs is connected.
For this consider two permutations~$\pi$ and~$\sigma$ that differ in cyclically shifting three consecutive elements, w.l.o.g.\ $\pi=12345\cdots n=\id$ and $\sigma=23145\cdots n$ (in particular, both permutations have even parity).
We can apply a single type-iv flip to $M_\pi$ to exchange the edges $(1,2)$ and $(3,4)$ for the edges $(1,4)$ and $(2,3)$, and we note that this yields a matching~$M_\sigma$.
It is well-known that shifting 3-cycles generate the alternating group, and thus any two matchings with the same parity are connected to each other.

The two components of the flip graph are isomorphic, as the entire flip graph is vertex-transitive.
\end{proof}

\begin{corollary}
The diameter of the flip graphs of directed perfect matchings of~$K_n$, where $n=2m$, under 2-flips of either type~i or type~ii (or both) is linear in~$m$.
\end{corollary}

\begin{proof}
Changing the orientation of one edge needs only a constant number of flips (recall Figures~\ref{fig:d2flips-typei} and~\ref{fig:d2flips-typeii}).
Adding or removing an edge requires at most two orientation changes plus one additional flip, i.e., again a constant number of flips.
This yields a linear upper bound on the diameter of the flip graph.

A linear lower bound is also easy to see, by considering two edge-disjoint matchings.
\end{proof}

\subsubsection{Hamiltonicity properties}

\begin{theorem}
\label{thm:CED-typei-lace}
Let $m\geq 4$ and~$n\coloneq 2m$.
The flip graph of directed perfect matchings of~$K_n$ under 2-flips of type~i is Hamilton-laceable.
\end{theorem}

We postpone the proof of Theorem~\ref{thm:CED-typei-lace} to after the following result and its proof.

\begin{theorem}
\label{thm:CED-typeii-lace}
Let $m\geq 3$ and~$n\coloneq 2m$.
The flip graph of directed perfect matchings of~$K_n$ under 2-flips of type~ii is Hamilton-laceable.
\end{theorem}

\begin{proof}
We argue by induction on~$m$.
For $m=3$ the flip graph has 120 vertices, and can be checked to be Hamilton-laceable with computer help.

For the induction step, let $m\geq 4$\footnote{In fact, the following inductive arguments already work under the assumption~$m\geq 3$, but this would not be meaningful, as the base case would be missing.}, and let $M$ and~$M'$ be two distinct directed perfect matchings of~$K_n$ with opposite parity.
We label the vertices $1,\ldots,n$ as follows:
If $M\triangle M'$ contains a directed cycle of length~2, we label its vertices $1,2$ such that $(1,2)\in M$ and $(2,1)\in M'$.
If $M\triangle M'$ contains no directed cycles of length~2, but only longer cycles, then let~$(1,2)\in M$ be an edge on one of them, and label the previous vertex on the cycle~3 such that either $(1,3)\in M'$ or $(3,1)\in M'$.
The remaining vertex labels~$3,\ldots,n$ or $4,\ldots,n$, respectively, are assigned arbitrarily.

We partition the set of all directed perfect matchings of~$K_n$ according to the edge starting or ending at vertex~1.
Specifically, for $i=2,\ldots,n$, let $\cM_i^+$ and~$\cM_i^-$ be the sets of matchings that contain the fixed edge~$(1,i)$ or~$(i,1)$, respectively.
We choose a total ordering $\cN=(\cN_1,\ldots,\cN_{2n-2})$ of those sets such that $M\in \cN_1$ and~$M'\in \cN_{2n-2}$, with the additional requirement that if $\cN_i=\cM_j^+$ for some $j\in[2,n]$ then $\cN_{i+1}\neq \cM_j^-$, and similarly if $\cN_i=\cM_j^-$ for some $j\in[2,n]$ then $\cN_{i+1}\neq \cM_j^+$ for all $i=1,\ldots,2n-3$.
We then choose pairs of matchings $A_i,B_i\in\cN_i$ for $i=1,\ldots,2n-2$ satisfying the following conditions:
\begin{itemize}[topsep=1mm,leftmargin=4mm]
\item $A_1=M$ and $B_{2n-2}=M'$;
\item $A_i$ and $B_i$ have opposite parity for $i=1,\ldots,2n-2$ (in particular, $A_i\neq B_i$);
\item $A_{i+1}$ differs from $B_i$ by a type-ii flip for all $i=1,\ldots,2n-3$.
\end{itemize}
Note that each set~$\cN_i$ contains at least three matchings for which a type-ii flip leads to a matching in~$\cN_{i+1}$.
Specifically, if $\cN_i=\cM_j^+$ for some $j\in[2,n]$, then these are matchings in which the edge incident with the vertex~$k$ goes out of~$k$, whereas if $\cN_i=\cM_j^-$, then these are matchings in which the corresponding edge goes into~$k$ (the parity can be fixed by directing one of the remaining $m-2$ edges appropriately); see Figure~\ref{fig:AiBi}.
Thus, when $A_{2n-3}$ is chosen, there are at least two choices for~$B_{2n-3}$, and not both of their neighbors in~$\cN_{2n-2}$ can be equal to the prescribed last matching~$B_{2n-2}=M'$, so one of the neighbors can become~$A_{2n-2}$.

\begin{figure}[h!]
\includegraphics[page=4]{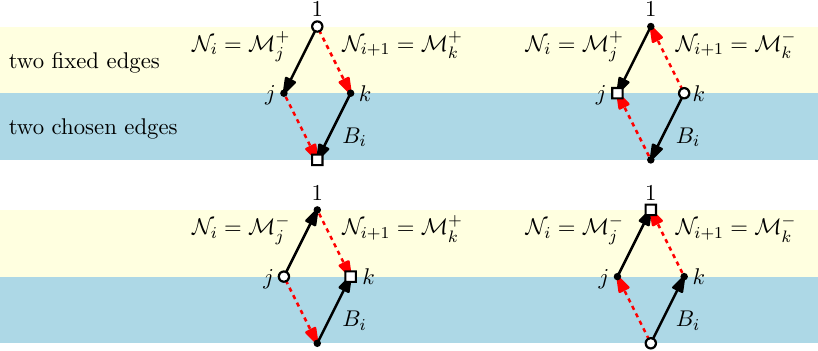}
\caption{Four different cases when selecting the matchings~$B_i$ in the proof of Theorem~\ref{thm:CED-typeii-lace}.
The top two edges are fixed by the order in which the sets~$\cM_j^+$ and~$\cM_j^-$ are visited, and the bottom two edges are chosen so that the transition can be realized with a type-ii flip.}
\label{fig:AiBi}
\end{figure}

By induction, for $i=1,\ldots,2n-2$ there is a path~$P_i$ in the flip graph from~$A_i$ to~$B_i$ through all matchings in~$\cN_i$.
The concatenation $P_1,P_2,\ldots,P_{2n-2}$ is the desired Hamilton path in the entire flip graph.
\end{proof}

\begin{proof}[Proof of Theorem~\ref{thm:CED-typei-lace}]
We argue by induction on~$m$.
For $m=4$ the flip graph has 1680 vertices, and can be checked to be Hamilton-laceable with computer help.

The proof of the induction step is analogous to the preceding proof of Theorem~\ref{thm:CED-typeii-lace}.
There is, however, one subtlety, namely flips of type~i are more restrictive than type-ii flips in that they only allow transitions~$\cM_i^+\leftrightarrow \cM_j^-$ for some $i,j\in[2,n]$ (shown in the top right and bottom left of Figure~\ref{fig:AiBi}), but never $\cM_i^+\leftrightarrow \cM_j^+$ or $\cM_i^-\leftrightarrow \cM_j^-$ (top left or bottom right in the figure, respectively).
In other words, we have to alternate between sets of matchings in which the fixed edge incident with vertex~1 goes out of it and those in which it goes into it.
This causes problems if both the start matching~$M$ and target matching~$M'$ are of the first type (+), namely if $(1,2)\in M$ (i.e., $M\in\cM_2^+$) and $(1,3)\in M'$ (i.e., $M'\in\cM_3^+$).
These problems can be resolved by splitting one of the intermediate sets~$\cM_i^+$, $i\in[2,n]$, into two subsets, thus allowing a consistent alternating pattern.
We omit the details.
\end{proof}

We also establish the following result, analogous to Theorem~\ref{thm:CE2-hconn}.

\begin{theorem}
\label{thm:CED-hconn}
Let $m\geq 2$ and~$n\coloneq 2m$.
The flip graph of directed perfect matchings of~$K_n$ under 2-flips of type~iv and type~v plus exactly one type from~$\{i,ii,iii\}$ is Hamilton-connected.
\end{theorem}

The set of flip types for achieving Hamilton-connectedness is minimal, as any subset of flip types from $\{i,ii,iii\}$ always yields a bipartite flip graph.

\begin{proof}
We first observe that the inductive argument given in the proof of Theorem~\ref{thm:CED-typeii-lace} goes through already for $m\geq 3$ and the stronger property of Hamilton-connectedness (replacing the parity condition on the~$A_i$ and~$B_i$ simply by $A_i\neq B_i$) for either of the following types of flips: type~ii, type~iv or type~v.
These three types of flips have in common that in the alternating 4-cycles shown in Figure~\ref{fig:dir-2flips}, each of the four possible combinations, namely in-edge/out-edge and matching edge/non-matching edge appears at exactly one of the 4 vertices.
This means we can choose the bottom two edges in Figure~\ref{fig:AiBi} appropriately for each of these flip types, so that all of the required transitions can be realized.

Thus it is enough to check the three base cases for $m=2$, either by hand using Figure~\ref{fig:G4d}~(a) or with computer help.
\end{proof}

\subsection{Directed almost-perfect matchings}

We now consider \emph{directed} almost-perfect matchings of the complete graph with an odd number of vertices.
In the undirected case, there was a one-to-one correspondence between perfect matchings of~$K_{2m}$ and almost-perfect matchings of~$K_{2m-1}$, given by removing a fixed vertex.
In the directed case, the same operation is a two-to-one correspondence, because the direction of the removed edge is `forgotten'.
Consequently, the number of directed almost-perfect matchings of~$K_n$, $n=2m+1$, is $2^m n!!$ (OEIS~A000407).

There are four different types of 1-flips, shown in Figure~\ref{fig:dir-1flips}.

\begin{figure}[h!]
\includegraphics[page=5]{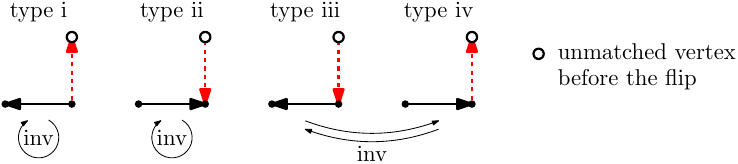}
\caption{The four different types of directed 1-flips.
The solid edge is in the original matching, and the dashed edge is in the target matching.}
\label{fig:dir-1flips}
\end{figure}

Flips of type~i and~ii are self-inverse, whereas type~iii and~iv are inverse to each other.

\begin{figure}[h!]
\includegraphics[page=3]{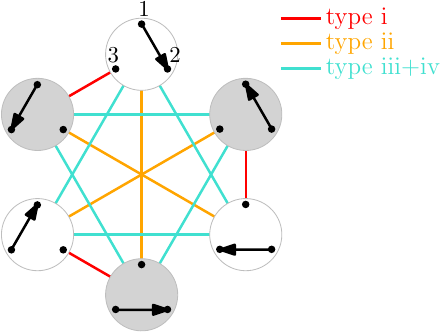}
\caption{Flip graph of directed almost-perfect matchings of~$K_3$, with edges colored according to the different flip types.
For type~iii and~iv, the same color is used, and each of these two types corresponds to one of the two orientations of the edge.
Matchings with even and odd parity are non-shaded and shaded, respectively.}
\label{fig:G3d}
\end{figure}

\subsection{Basic properties}

For a permutation~$\pi=(\pi_1,\ldots,\pi_n)$ of $[n]=[2m-1]$, we let $M_\pi$ be the directed almost-perfect matching with the edges~$(\pi_1,\pi_2),(\pi_3,\pi_4),\ldots,(\pi_{n-2},\pi_{n-1})$, leaving the vertex~$\pi_n$ unmatched.
Note that exactly $(m-1)!$ permutations~$\pi$ of~$[n]$ correspond to the same matching~$M_\pi$, namely those that differ only in the ordering of the pairs~$(\pi_{2i-1},\pi_{2i})$, $i\in[m-1]$.
All of those permutations have the same parity, which defines the \defi{parity} of the matching.

\begin{theorem}
\label{thm:COD-basic}
Let $m\geq 1$ and~$n\coloneq 2m+1$.
The flip graph of directed almost-perfect matchings of~$K_n$ under 1-flips of exactly one of the four types has the following properties:
\begin{itemize}[topsep=1mm,leftmargin=4mm]
\item type~i or ii: bipartite and has exactly $\binom{n}{m}=\binom{2m+1}{m}$ isomorphic connected components;
\item type~iii or iv: non-bipartite and has exactly two isomorphic connected components.
\end{itemize}
\end{theorem}

\begin{proof}
The proof is split into cases according to the different types of flips.

{\bf Type-i and type-ii flips:}
Reversing the direction of all edges in the matchings is an isomorphism between the two flips graphs.
Thus, for the rest of the proof it is sufficient to consider only type-i flips.
Note that under this type of flip, every vertex retains the property of being incident with an out-edge.
Thus, the entire flip graph has $\binom{n}{m}$ connected components, one for each choice of which $m$ vertices out of the $n$ vertices have out-edges, and each of these components is isomorphic to the flip graph for almost-perfect matchings of~$K_{m,m+1}$ discussed in Theorem~\ref{thm:COB-lace}.

{\bf Type-iii and type-iv flips:}
As they are inverse to each other, the two corresponding flip graphs differ only in the orientation of their edges.
Thus, for the rest of the proof it is sufficient to consider only type-iii flips.

As can be see from Figure~\ref{fig:G3d}, the flip graph has triangles, so it is not bipartite.

To argue that the flip graph has at least two connected components, consider a permutation~$\pi$ associated to the matching.
A type-iii flip in the matching changes a pair of entries~$(\pi_{2i-1},\pi_{2i})$ and the last entry~$\pi_n$ to $(\pi_n,\pi_{2n-1})$ and~$\pi_{2i}$, respectively, which can be seen as a sequence of two transpositions.
Thus the flip graph has at least two connected components, induced by all matchings with even or odd parity.

\begin{figure}[b!]
\includegraphics[page=6,scale=0.8]{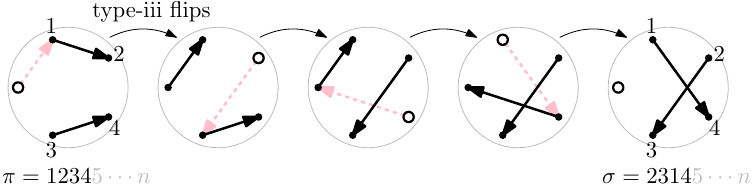}
\caption{Sequence of type-iii flips to cyclically shift three consecutive elements in the permutation.}
\label{fig:d1flips-typeiii}
\end{figure}

It remains to show that each of these two subgraphs is connected.
For $m=1$ this can be checked by consulting Figure~\ref{fig:G3d}.
For $m\geq 2$ we consider two permutations~$\pi$ and~$\sigma$ that differ in cyclically shifting three consecutive elements, w.l.o.g.\ $\pi=12345\cdots n=\id$ and $\sigma=23145\cdots n$.
Figure~\ref{fig:d1flips-typeiii} shows a sequence of four type-iii flips applied to the matching~$M_\pi$ to exchange the edges $(1,2)$ and $(3,4)$ for the edges $(1,4)$ and $(2,3)$, and we note that this yields a matching~$M_\sigma$.
It is well-known that shifting 3-cycles generate the alternating group, and thus any two matchings with the same parity are connected to each other.
\end{proof}

We note that for any combination of two out of four types, except for the combination~i+ii, the corresponding flip graph is connected.
Indeed, one can check that the two types of flips are enough to reverse a single edge.

\subsection{Hamiltonicity properties}

We also establish the following result, analogous to Theorem~\ref{thm:CO1-hconn}.

\begin{theorem}
\label{thm:COD-hconn}
Let $m\geq 1$ and~$n\coloneq 2m+1$.
The flip graph of directed almost-perfect matchings of~$K_n$ under 1-flips of all four types is Hamilton-connected.
\end{theorem}

\begin{figure}[b!]
\makebox[0cm]{ 
\includegraphics[page=1]{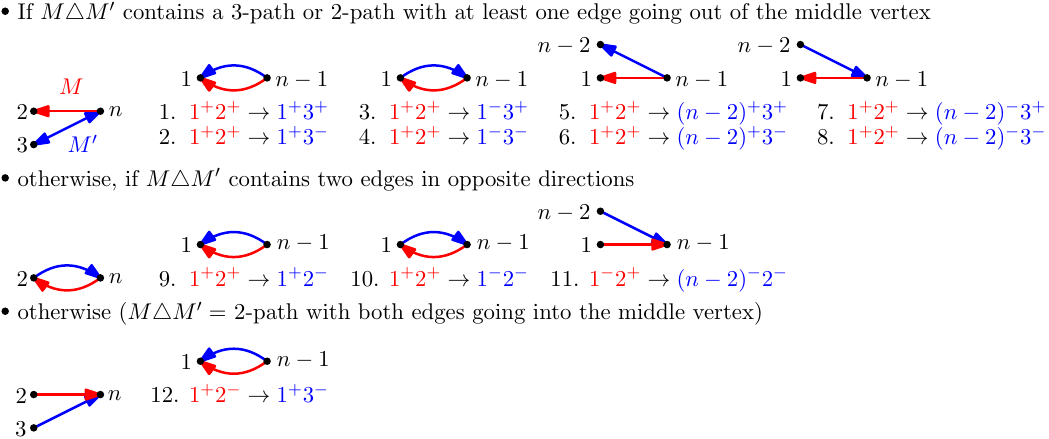}
}
\caption{Labeling of vertices in twelve different cases for the symmetric difference~$M\triangle M'$.}
\label{fig:dir-diff}
\end{figure}

\begin{proof}
\begin{figure}[t!]
\includegraphics[page=2,scale=0.85]{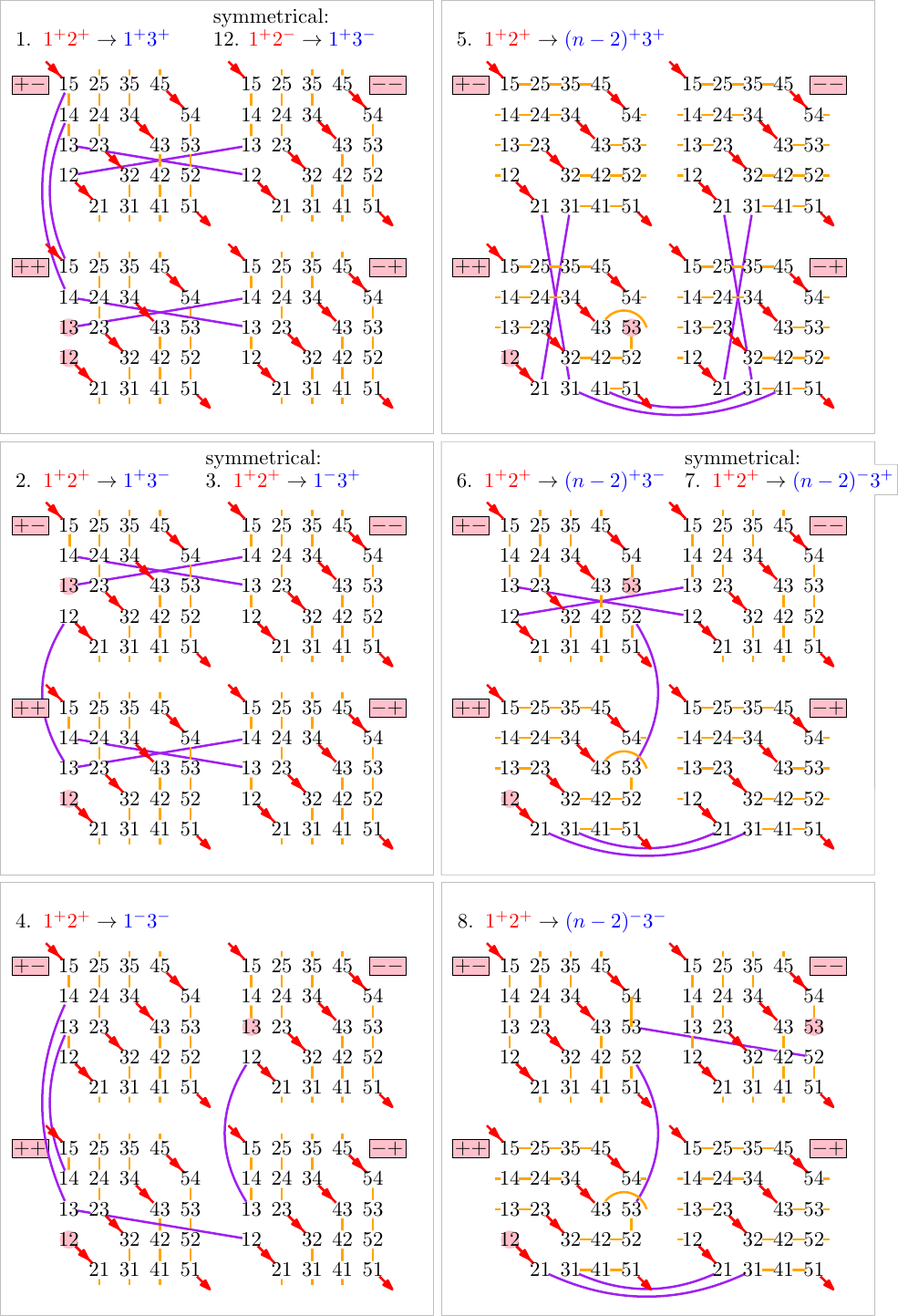}
\caption{Interleaving of pair connections in cases 1--8+12 from Figure~\ref{fig:dir-diff}.}
\label{fig:tableau1}
\end{figure}

\begin{figure}[t!]
\includegraphics[page=3,scale=0.85]{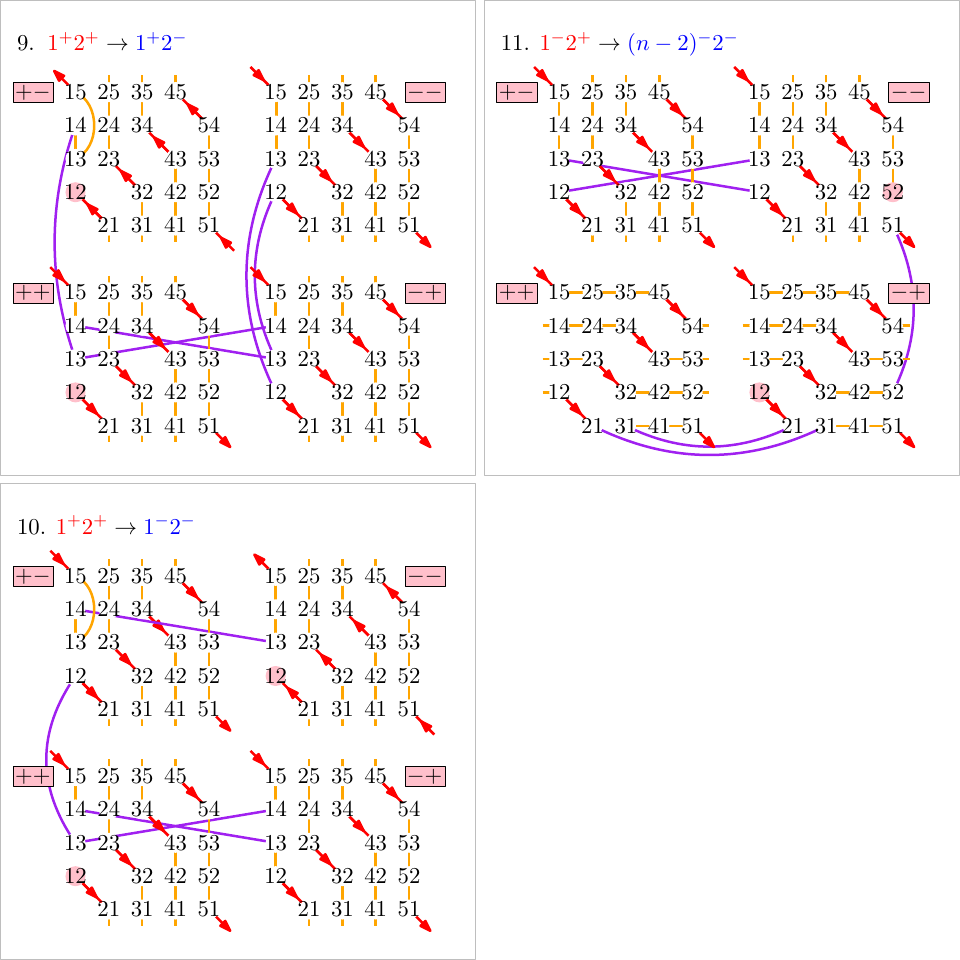}
\caption{Interleaving of pair connections in cases 9--11 from Figure~\ref{fig:dir-diff}.}
\label{fig:tableau2}
\end{figure}

The proof follows the same overall strategy as the proof of Theorem~\ref{thm:CO1-hconn} for the undirected setting presented in Section~\ref{sec:CO1-hconn}.
Namely, we use induction on~$m$.
The base cases~$m=1,2,3$ can be checked with computer help (the flip graphs have 6, 60 and 840 vertices, respectively).
For the induction step, let $m\geq 4$.
The induction step goes back by 2~steps from~$m$ to~$m-2$.
We start by suitably labeling the vertices~$1,\ldots,n$, where we distinguish twelve cases for the symmetric difference~$M\triangle M'$, as shown in Figure~\ref{fig:dir-diff}.

The set of all matchings is partitioned into sets~$\cX_{i,j}$, $\cY_i$, $\cY_i'$, $\cZ_i$, similarly to before, but now we also distinguish the direction of the two edges that are incident with the vertices~$n$ and~$n-1$ (or their neighbor~$i+1$; recall Figure~\ref{fig:tick2}) that are removed in the induction step.
Consequently, each of these sets comes in four varieties, indicated by two superscripts from the set~$\{+,-\}$ that specify the direction of the two relevant edges.
I.e., we have the sets $\cX_{i,j}^{++},\cX_{i,j}^{+-},\cX_{i,j}^{-+}$ and $\cX_{i,j}^{--}$, and similarly for the other three, yielding sixteen sets in total.
The rest of the proof is analogous to before.
Specifically, in each case, we traverse four copies of the listing of pairs that specifies the end vertices of the edges incident with~$n$ and~$n-1$.
The corresponding listings of pairs are shown in Figures~\ref{fig:tableau1} and~\ref{fig:tableau2} for each of the twelve cases.
Three of the twelve cases are symmetric, so only nine traversals are shown in the figures.
As before, these traversals use every diagonal edge~$\{(i,i+1),(i+1,i)\}$ from each variety, and when traversing such an edge we interleave the corresponding matchings in which vertex~$n-1$ or~$n$ are unmatched.
\end{proof}

\section{Geometric setting}
\label{sec:geom}

In this section we consider the setting of (undirected) straight-line non-crossing matchings on a set of points in the plane.
We mostly consider points in convex position, only the very last result (Theorem~\ref{thm:GOP}) is about points in general position.

\subsection{Non-crossing perfect matchings}

We first recap the results on perfect matchings on a convex set of points obtained by Hernando, Hurtado and Noy~\cite{MR1939072}, already discussed in the introduction.

\begin{theorem}[{\cite[Thm.~5.3]{MR1939072}}]
\label{thm:GE-even}
Let $m\geq 4$ be even and $n\coloneq 2m$.
The flip graph of non-crossing perfect matchings on $n$ points in convex position under 2-flips has a Hamilton cycle.
\end{theorem}

For $m=2$ the flip graph is a single edge, and for $m=3$ it is $K_{2,3}$ (see Figure~\ref{fig:settings}~(c)), so in both cases there is a Hamilton path.

They also proved that the flip graph is bipartite (for all~$m$), and that it has hugely imbalanced partition classes for odd~$m$, leading to the following result.

\begin{theorem}[{\cite[Lem.~5.1+Thm.~5.2]{MR1939072}}]
\label{thm:GE-odd}
Let $m\geq 5$ be odd and $n\coloneq 2m$.
In the flip graph of non-crossing perfect matchings on $n$ points in convex position under 2-flips, any path misses at least $C_{(m+1)/2-1}-1=\Theta(2^m/m^{3/2})$ many vertices. 
In particular, there is no Hamilton path (nor cycle).
\end{theorem}

\subsection{Non-crossing almost-perfect matchings}

We consider $n$~points in convex position, labeled $1,\ldots,n$ in clockwise order.
For some integer~$m\geq 1$, we define $n\coloneq 2m+1$, and we let~$\cM_m$ be the set of non-crossing almost-perfect matchings on those $n$ points, leaving exactly one of the points unmatched.
Furthermore, we define $n'\coloneq n-1=2m$ and let~$\cP_m$ be the set of non-crossing perfect matchings on $n'$ points in convex position.
A matching edge is \defi{visible} from some point if the triangle spanned by the point and the two edge endpoints does not intersect any other matching edges.
The \defi{length} of a matching edge~$e=\{i,j\}$ is defined as the minimum number of points on either of its two sides, plus~1.
For example, the length of the edge~$e=\{2,5\}$ in a matching~$M$ is $\ell(e)=3$ if $M\in\cM_4$, and it is $\ell(e)=2$ if $M\in\cM_2$.
We sometimes refer to an edge of length~$\ell$ as an \defi{$\ell$-edge}.

We start by counting these matchings.

\begin{lemma}
We have $|\cM_m|=n |\cP_m|=n C_m$.
\end{lemma}

\begin{proof}
It is well-known that non-crossing perfect matchings with $m$ edges are counted by the Catalan numbers~$C_m$.
To each such perfect matching we can insert one additional unmatched $n$th point in a fixed position, and rotate the resulting matching in all $n$ possible ways.
\end{proof}

Given a matching~$M\in\cM_m$, its degree in the flip graph is twice the number of matching edges that are visible from the unmatched point.
In particular, if the unmatched point is~$i$ and the edge~$\{i-1,i+1\}$ is present in~$M$, then $M$ has degree~2 in the flip graph.
One can check that the number of such matchings is~$n C_{m-1}$.
Consequently, the fraction of degree-2 vertices among all vertices is~$C_{m-1}/C_m=1/4(1+o(1))$, i.e., converging to~$1/4$ as $m$ grows.
This imposes severe constraints on finding a Hamilton path or cycle.

\begin{figure}[b!]
\includegraphics[page=1,scale=0.7]{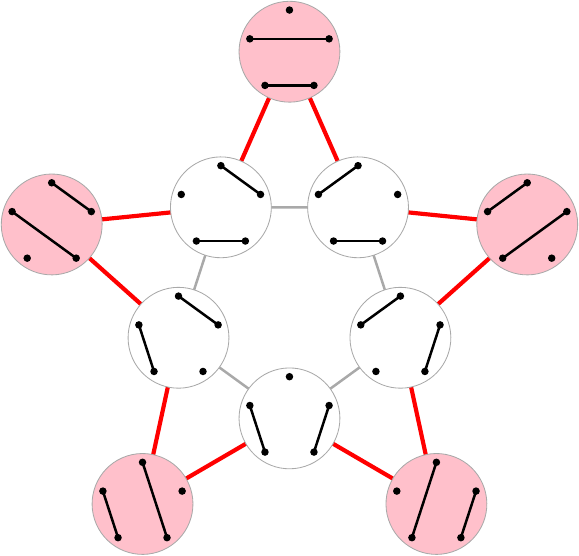}
\caption{Flip graph of non-crossing matchings on 5 points ($m=2$ edges), with degree-2 vertices and their incident edges highlighted.}
\label{fig:G5geom}
\end{figure}

In particular, it was observed by Aichholzer, Dorfer, Rieck and Verciani~\cite{booklet-problem} that there is a set of $n$ degree-2 vertices in the flip graph whose incident edges link up to a cycle of length~$2n$; see Figures~\ref{fig:G5geom} and~\ref{fig:G7geom}.
Specifically, these are matchings containing a single 2-edge and all other edges have length~1 (i.e., they lie on the boundary).
For $m=2$ this cycle of length~$2n$ is a Hamilton cycle, but for larger~$m$ it is not, so a Hamilton cycle in the flip graph is ruled out for~$m\geq 3$.

In fact, we prove the following much stronger result.

\begin{theorem}
\label{thm:GO-miss}
For integers $a\geq 5$ and $0\leq b\leq 4$, let $m\coloneq 9+5a+b$ and $n\coloneq 2m+1$.
In the flip graph of non-crossing almost-perfect matchings on $n$ points in convex position under 1-flips, any path misses at least $2^{m/5}/16=\Theta(2^{m/5})$ many vertices.
In particular, there is no Hamilton path (nor cycle).
\end{theorem}

\begin{figure}[t!]
\includegraphics[page=3,scale=0.7]{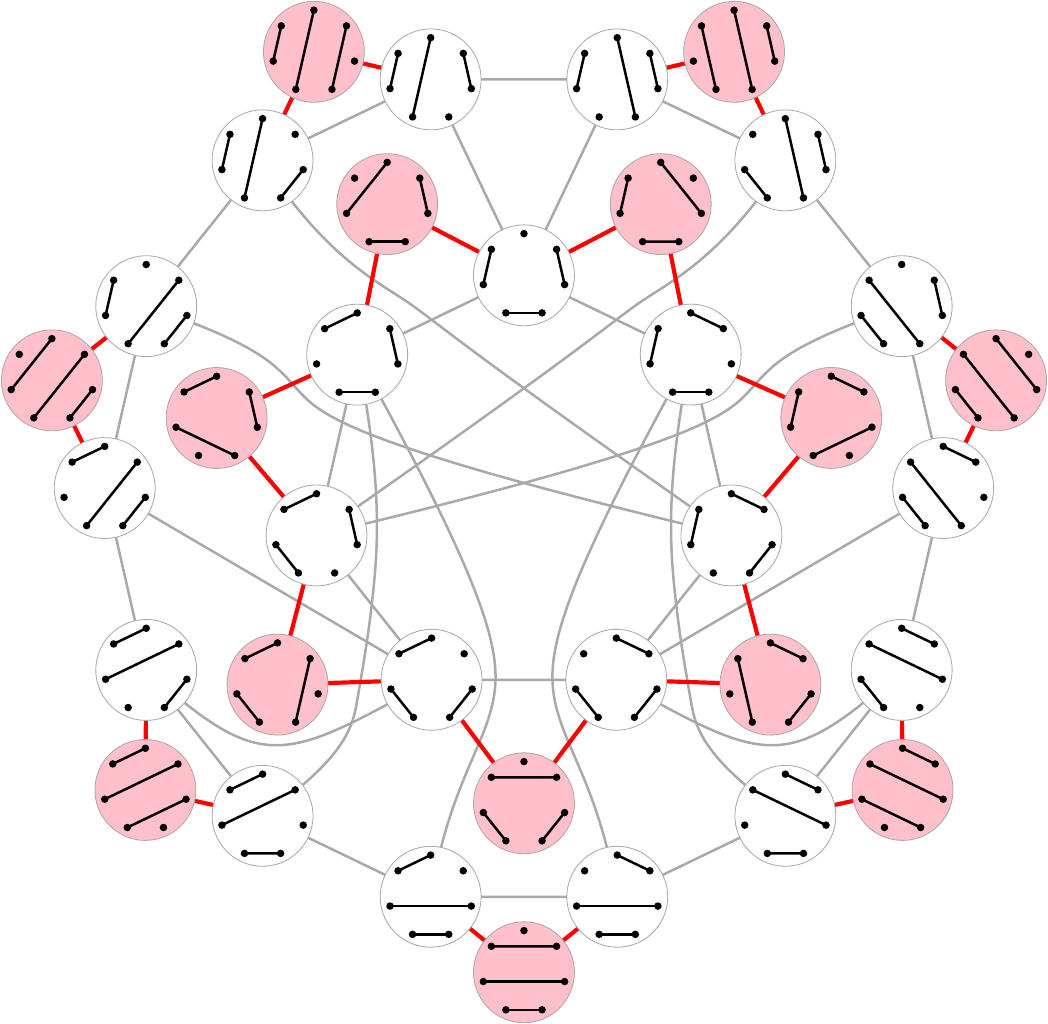}
\caption{Flip graph of non-crossing matchings on 7 points ($m=3$ edges), with degree-2 vertices and their incident edges highlighted.}
\label{fig:G7geom}
\end{figure}

Note that the previous theorem starts to kick in only for relatively large values of~$m$, and probably the proof could be optimized to decrease those a little bit.
In fact, computer experiments show that a Hamilton path in the flip graph exists for $m=3,4,5,6,7$, but not for $m=8$.

Complementing Theorem~\ref{thm:GO-miss}, we can prove the existence of long cycles in the flip graph, namely cycles that visit almost all vertices, i.e., a $(1-o(1))$-fraction as $n$ tends to infinity.

\begin{theorem}
\label{thm:GO-long}
Let $m\geq 1$ and $n\coloneq 2m+1$.
The flip graph of non-crossing almost-perfect matchings on~$n$ points in convex position under 1-flips has a cycle that visits at least a $(1-14/n)$-fraction of all vertices.
\end{theorem}

As mentioned before, Theorem~\ref{thm:GO-long} does not contradict Theorem~\ref{thm:GO-miss}, because the fraction of missed vertices in Theorem~\ref{thm:GO-miss} is only~$o(1)$, in fact exponentially decreasing, whereas $14/n$ is polynomially decreasing.

The proof of Theorem~\ref{thm:GO-long} follows a 2-step strategy that has been used successfully in several previous papers.
Specifically, we first build a \defi{cycle factor}, i.e., a collection of disjoint cycles that together visit all vertices in the flip graph.
We then glue these cycles together to a single cycle.
In the gluing step, we lose a few vertices that do not get included into the final cycle.

In the remainder of this section, we prove Theorems~\ref{thm:GO-miss} and~\ref{thm:GO-long}.
We start with the proof of Theorem~\ref{thm:GO-long}, and will reuse some of the terminology also in the proof of Theorem~\ref{thm:GO-miss} presented afterwards.

\subsubsection{More notation and terminology on non-crossing matchings}

Given a matching~$M\in\cM_m$ or~$M\in\cP_m$, we write $r(M)$ for the matching obtained from~$M$ by rotating it one step in counterclockwise direction.
For a perfect matching~$M\in\cP_m$, we write~$\ord(M)$ for the smallest integer~$i>0$ such that $r^i(M)=M$.
Note that $\ord(M)\geq 2$ and that $\ord(M)$ divides~$n'$.

We write $M^1\in\cP_m$ for the matching that has all edges of length~1, namely the edges~$\{1,2\},\allowbreak\{3,4\},\ldots,\{n'-1,n'\}$.
Note that its order is $\ord(M^1)=2$.
The two matchings in~$[M^1]$ are the only ones whose order is~2, whereas all other matchings have order at least~3.

For two matchings~$M,M'\in\cM_m$, we write $M\sim M'$ if $M'=r^i(M)$ for some integer~$i\geq 0$.
This defines an equivalence relation on the set~$\cM_m$, and we write $[M]\coloneq \{M'\in\cM_m\mid M'\sim M\}=\{r^i(M)\mid i\geq 0\}$ for the equivalence class of~$M$ under rotation, and $\cM_m/{\sim}\coloneq \{[M]\mid M\in\cM_m\}$ for the set of equivalence classes.
The same notations are also defined analogously for perfect matchings~$\cP_m$.
For any matching $M\in\cM_m$ each equivalence class has size $|[M]|=n$, whereas for any perfect matching $M\in\cP_m$ the size is $|[M]|=\ord(M)$.

Given a matching~$M\in\cM_m$ with unmatched point~$i\in[n]$, we write $p(M)\in\cP_m$ for the perfect matching obtained by removing the unmatched point, decreasing the labels of all points $i+1,\ldots,n$ by~1 to become $i,\ldots,n-1$.
Conversely, given a perfect matching~$M\in\cP_m$, we write $c_i(M)\in\cM_m$, $i=1,\ldots,n$, for the matching obtained by increasing the labels of all points~$i,\ldots,n-1$ by~1 to become $i+1,\ldots,n$ and then inserting an unmatched point~$i$.
Note that if the point~$i$ is unmatched in~$M\in\cM_m$, then $c_i(p(M))=M$.
Also note that $c_n(M)=r(c_1(M))$ for all $M\in\cP_m$.

A \defi{plane tree} is a tree in which every vertex has a circular ordering of its neighbors.
We think of it as a tree embedded in the plane where the circular ordering is given by the ordering of neighbors around every vertex, in clockwise direction, say.
We write~$\cT_m$ for the set of plane trees with $m$ edges.
The number plane trees with $m\geq 1$ edges is $1,1,2,3,6,14,34,95,280,854,\ldots$ (OEIS A002995), which satisfies $|\cT_m|=(1+o(1))4^m/(\sqrt{\pi}m^{5/2})$.
Plane trees with $m$ edges are in one-to-one correspondence to equivalence classes of non-crossing perfect matchings with $m$ edges, i.e., we have $|\cT_m|=|\cP_m/{\sim}|$.
Specifically, for any equivalence class $[M]\in\cP_m/{\sim}$, the corresponding plane tree is the dual graph of (any rotation of) the matching~$M$, without a dual vertex for the outer face; see Figure~\ref{fig:cfac}.

\subsubsection{Cycle factor construction}
\label{sec:cfac}

Given a matching~$M\in\cM_m$ in which point~$i\in[n]$ is unmatched, we write $f(M)\in\cM_m$ for the matching obtained from~$M$ by unmatching the next point~$i+1$ (in clockwise direction), i.e., $f(M)\coloneq u(M,i+1)$.
Formally, if $j$ is the neighbor of~$i+1$ in~$M$, we replace the edge $\{i+1,j\}$ by~$\{i,j\}$.
This rule is clearly invertible, and we thus obtain a collection of disjoint cycles that together visit all vertices of the flip graph; see Figure~\ref{fig:cfac}.
Specifically, the cyclic sequence containing the matching~$M$ is defined as
\begin{subequations}
\label{eq:cfac}
\begin{equation}
C(M)\coloneq (f^i(M))_{i\geq 0},
\end{equation}
and the cycle factor as
\begin{equation}
\cC_m\coloneq \{C(M)\mid M\in\cM_m\}.
\end{equation}
\end{subequations}

\begin{figure}[t!]
\includegraphics[page=1,scale=0.65]{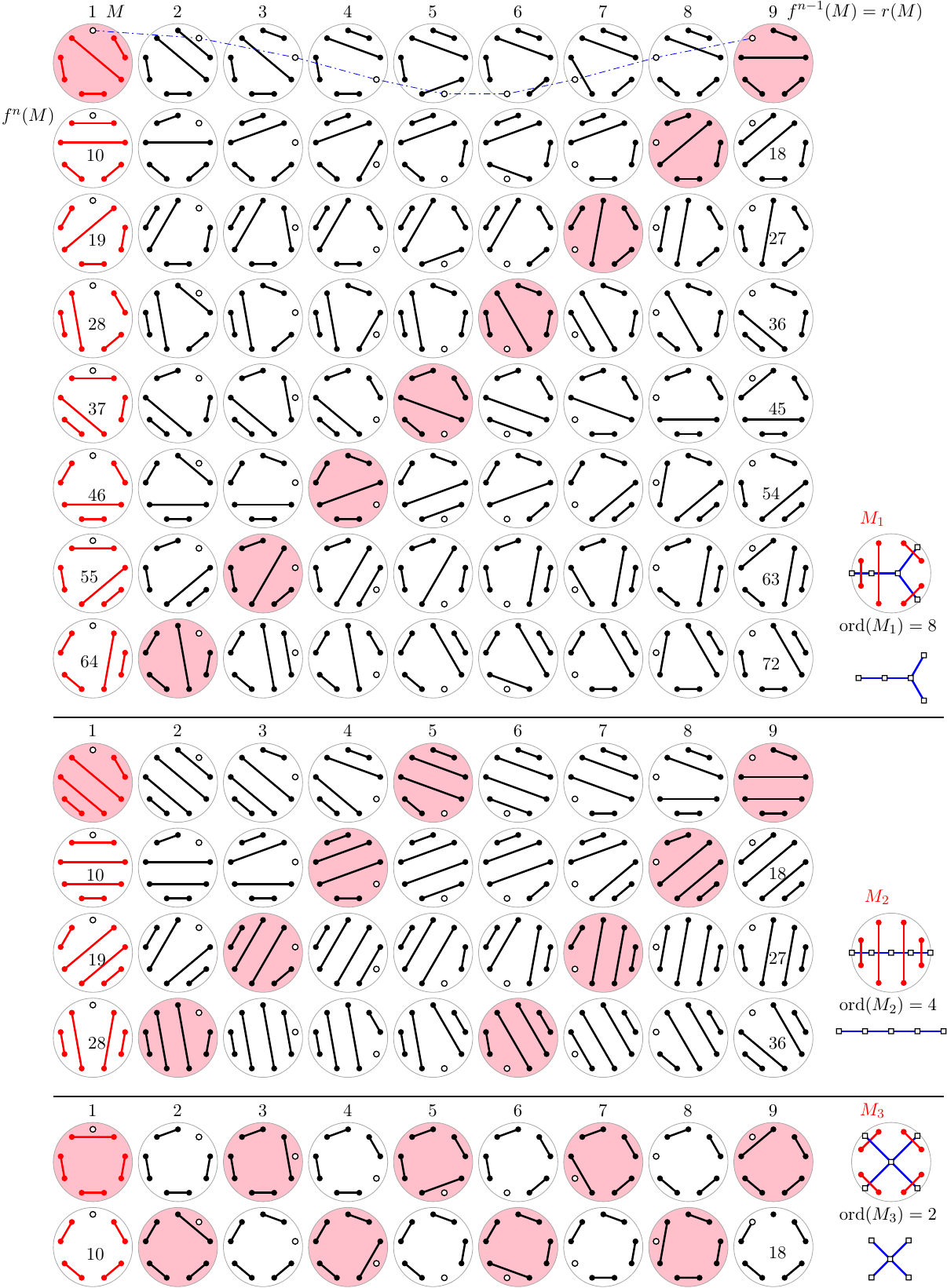}
\caption{Construction of the cycle factor~$\cC_4$, which has three cycles of lengths 72, 36, 27, corresponding to the three plane trees with 4~edges shown on the right (of order~8, 4 and 2, respectively).
The 9 cyclic rotations of the first matching in each cycle are highlighted in red.
In column~$i=1,\ldots,9$, the unmatched point is~$i$, and going down corresponds to rotating the matching on the remaining $n-1$ points in counterclockwise direction.
}
\label{fig:cfac}
\end{figure}

Clearly, if point~$i\in[n]$ is unmatched in~$M$, then in~$f^j(M)$ the point~$i+j$ is unmatched (all indices are considered modulo~$n$).
In particular, in $M'\coloneq f^n(M)$, the same point~$i$ is unmatched again.
Furthermore, the matching~$M'$ differs from~$M$ by rotating the matching on the remaining $n-1$ points by one step in counterclockwise direction (ignoring the unmatched point~$i$), i.e., we have $p(M')=r(p(M))$.
Consequently, we have $f^{n-1}(M)=r(M)$.
It follows that $[M]\subseteq C(M)$, and that the cycles of~$\cC_m$ are in one-to-one correspondence to equivalence classes of perfect matchings from~$\cP_m$ under rotation.

\begin{lemma}
\label{lem:cfac}
For any matching~$M\in\cM_m$, the cycle $C(M)$ contains exactly all matchings obtained by rotating~$p(M)$ arbitrarily and inserting one additional unmatched point.
Consequently, the cycle~$C(M)$ has length~$n \ord(p(M))$.
The cycles in the factor~$\cC_m$ are in bijection with equivalence classes of perfect matchings in~$\cP_m$ under rotation, or equivalently, with plane trees in~$\cT_m$, i.e., we have $|\cC_m|=|\cT_m|$.
\end{lemma}

Based on this lemma we may also define $C(p(M))\coloneq C(M)$ for all~$M\in\cM_m$.

\subsubsection{Gluing cycles together}

An \defi{alternating cycle} in the flip graph is a cycle that alternately uses one edge on a cycle of the factor~$\cC_m$ and one edge between two of these cycles.
Thus, taking the symmetric difference of the edge sets of the alternating cycle and the cycles of the factor glues them together to a single cycle.
A shortest possible alternating cycle has length~4, i.e., a 4-cycle that has two opposite edges on the cycles and the other two edges between the cycles.
Unfortunately, the flip graph of non-crossing matchings does not have any 4-cycles.
Instead, we use longer cycles for gluing pairs of cycles from the factor together, at the cost of losing some vertices from the cycles in the process.
Figure~\ref{fig:gluing} illustrates this operation schematically (see also Figure~\ref{fig:ear}~(d)).

\begin{figure}[h!]
\includegraphics{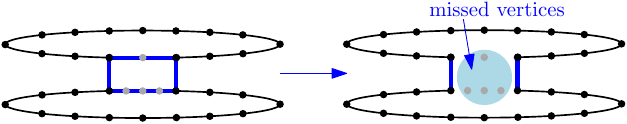}
\caption{Illustration of the gluing step.}
\label{fig:gluing}
\end{figure}

An \defi{ear} of a perfect matching~$M\in\cP_m$ is an ordered pair~$(i,d)$ with $1\leq i\leq n'$ and $3\leq d\leq m$, such that $\{i,j\}$ for $j\coloneq i+d$ (modulo~$n'$) is an edge of~$M$ of length $\ell(\{i,j\})=d$, and all other edges of~$M$ on the points~$i+1,i+2,\ldots,j-1$ have length~1; see Figure~\ref{fig:ear}~(a).
We refer to $d$ as the \defi{length} of the ear~$(i,d)$.
Furthermore, we refer to the edge $\{i,j\}$ of~$M$ as an \defi{ear edge}, and we write $E(M)$ for the set of all ear edges in~$M$.

Note that every 3-edge~$\{i,i+3\}$ of a matching is an ear edge, as the two points~$i+1,i+2$ admit only one edge between them, namely the 1-edge~$\{i+1,i+2\}$.
Furthermore, observe that every matching~$M\in\cP_m\setminus[M^1]$ has at least one ear edge, consider for example any shortest edge of length~$>1$.

Given a perfect matching~$M\in\cP_m$ with an ear~$e=(i,d)$, we write $M-e\in\cP_m$ for the matching obtained from~$M$ by rotating the sub-matching on the points~$i,i+1,\ldots,j$, where $j\coloneq i+d$, by one step, i.e., in~$M-e$ all edges on the points~$i,i+1,\ldots,j$ have length~1; see Figure~\ref{fig:ear}~(b).
We refer to this operation as \defi{ear removal}.

\begin{figure}[t!]
\includegraphics{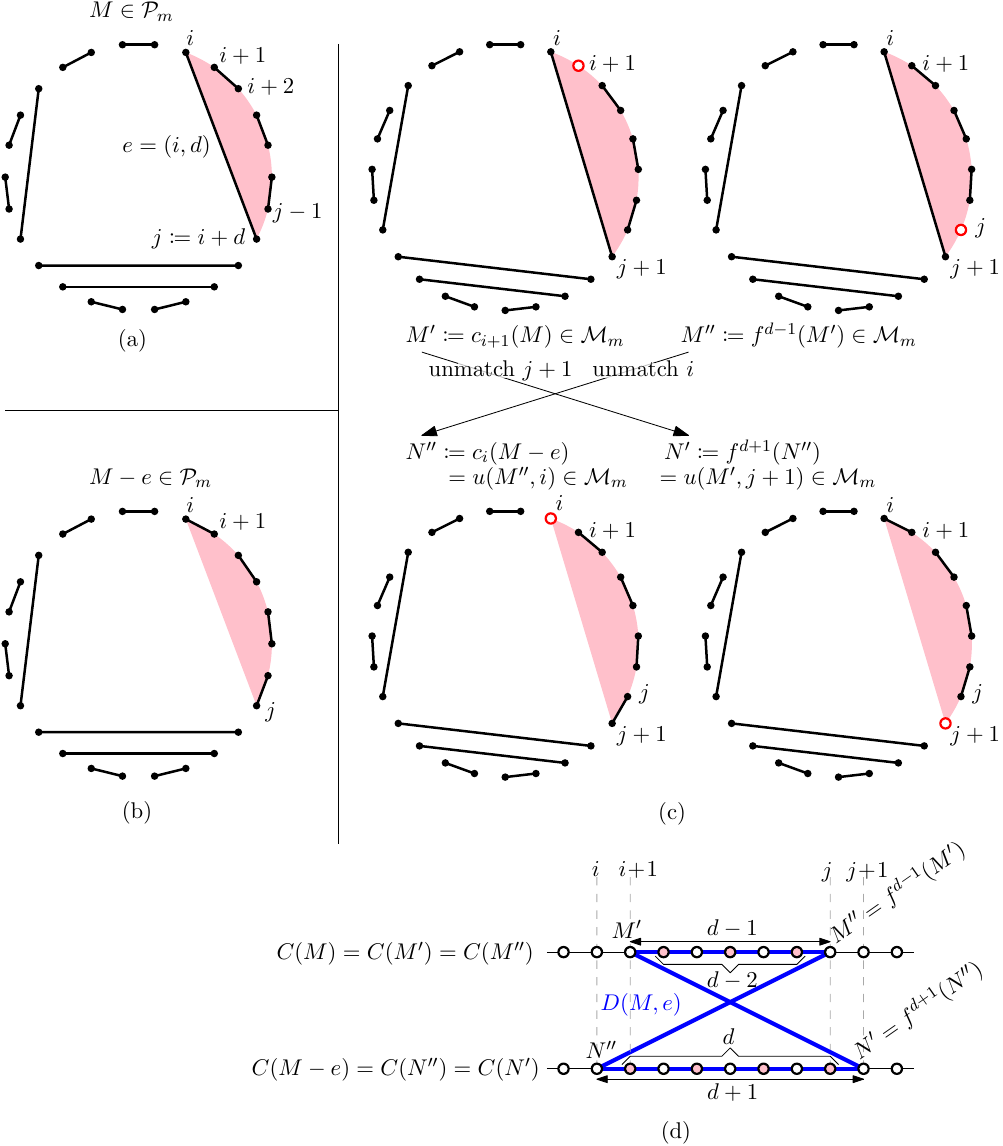}
\caption{Illustration of (a) an ear $e=(i,d)$; (b) of the ear removal operation~$M-e$; (c)+(d) the corresponding gluing cycle~$D(M,e)$.
}
\label{fig:ear}
\end{figure}

Given a matching~$M\in\cP_m$ and an ear~$e=(i,d)$ of~$M$, consider the matchings~$M'\coloneq c_{i+1}(M)\in\cM_m$ and~$M''\coloneq f^{d-1}(M')$; see Figure~\ref{fig:ear}~(c).
These matchings lie in distance~$d-1$ on the cycle~$C(M)=C(M')=C(M'')$.
Now consider the matchings~$N''\coloneq c_i(M-e)\in\cM_m$ and~$N'\coloneq f^{d+1}(N'')\in\cM_m$, which are in distance~$d+1$ on the cycle~$C(M-e)=C(N'')=C(N')$.
Note that $N''$ and~$N'$ can also be obtained from~$M''$ and~$N'$, respectively, by unmatching the point~$i$ or~$j+1$.
Consequently, there is a cycle~$D(M,e)$ containing the edges between~$M'$ and~$M''$ on $C(M')$, the edges between~$N''$ and~$N'$ on~$C(N'')$, plus the two edges~$\{M',N'\},\{M'',N''\}$.
We refer to~$D(M,e)$ as a \defi{gluing cycle}.
The length of the cycle~$D(M,e)$ is~$(d-1)+(d+1)+2=2d+2$.
The symmetric difference of the edge sets of the cycles~$C(M')$ and~$C(N'')$ with the cycle~$D(M,e)$ yields one cycle plus two paths with $d-2$ and~$d$ vertices, i.e., the total number of vertices on one of the original cycles but not on the resulting cycle is~$(d-2)+d=2d-2$.

\subsubsection{Connecting cycles along a spanning tree}

We say that an ear~$e=(i,d)$ of~$M$ of length~$d\geq 5$ is \defi{open}, if $\{i-4,i-3\},\{i-2,i-1\}\in M$ and $\{j+1,j+2\},\{j+3,j+4\}\in M$ where $j\coloneq i+d$, \defi{half-open} if exactly one of these two conditions holds, and \defi{closed} if neither of them holds.

Next, we choose a set $\cP_m'\subseteq \cP_m\setminus[M^1]$ of representatives, exactly one from each equivalence class of~$(\cP_m\setminus[M^1])/{\sim}$, arbitrarily.
We then define a mapping
\begin{subequations}
\label{eq:tau}
\begin{equation}
\tau:\cP_m'\rightarrow \cP_m,
\end{equation}
as follows.
For a perfect matching~$M\in\cP_m'$, we consider all ears of maximum length in~$M$.
If the maximum length is~3 or the maximum length is at least~5 and all such ears are closed, let $e=(i,d)$ be one of them, and define
\begin{equation}
\label{eq:tau1}
\tau(M)\coloneq M-e.
\end{equation}
If the maximum length is at least~5 and one of those ears is open or half-open, let $e=(i,d)$ be one of them, and proceed as follows:
If $\{i-4,i-3\},\{i-2,i-1\}\in M$, we define
\begin{equation}
\begin{aligned}
\label{eq:tau2}
\tau(M)&\coloneq M' \text{ for } M'\coloneq (M\setminus\{\{i-4,i-3\},\{i-2,i-1\}\})\cup\{\{i-4,i-1\},\{i-3,i-2\}\}.
\end{aligned}
\end{equation}
Note that $f\coloneq(i-4,3)$ is an ear of length~3 of~$M'$ and we have $M'-f=M$.
Similarly, if $\{j+1,j+2\},\{j+3,j+4\}\in M$ where $j\coloneq i+d$, we define
\begin{equation}
\begin{aligned}
\label{eq:tau3}
\tau(M)&\coloneq M' \text{ for } M'\coloneq (M\setminus\{\{j+1,j+2\},\{j+3,j+4\}\})\cup\{\{j+1,j+4\},\{j+2,j+3\}\}. \\
\end{aligned}
\end{equation}
\end{subequations}
Note that $f\coloneq (j+1,3)$ is an ear of length~3 of~$M'$ and we have $M'-f=M$.
For open ears both conditions are satisfied and then we define~$\tau(M)$ according to~\eqref{eq:tau2}.

\begin{lemma}
\label{lem:sptree}
The set of pairs $([M],[\tau(M)])$ for all $M\in\cP_m'$ defines a spanning tree on the set of equivalence classes~$\cP_m/{\sim}$.
\end{lemma}

\begin{proof}
Recall that $E(M)$ denotes the set of ear edges of a matching~$M\in\cP_m$.
For any perfect matching~$M\in\cP_m$ we define a 5-tuple of non-negative integers
\begin{subequations}
\label{eq:phi}
\begin{equation}
\varphi(M)\coloneq (n_1,n_2,n_3,n_4,n_5)
\end{equation}
where
\begin{equation}
\begin{aligned}
n_1 &\coloneq \text{number of non-ear edges $e\in M\setminus E(M)$ with $\ell(e)\geq 5$}, \\
n_2 &\coloneq \text{number of ear edges $e\in E(M)$ with $\ell(e)\geq 5$}, \\
n_3 &\coloneq \text{number of ear edges $e\in E(M)$ with $\ell(e)\geq 5$ corresponding to open ears}, \\
n_4 &\coloneq \text{number of ear edges $e\in E(M)$ with $\ell(e)\geq 5$ corresponding to half-open ears}, \\
n_5 &\coloneq \text{number of ear edges $e\in E(M)$ with $\ell(e)=3$.}
\end{aligned}
\end{equation}
\end{subequations}

A cycle in a directed graph is called \defi{oriented} if all of its edges have the same orientation along the cycle, and \defi{mixed} otherwise.

We think of the pairs~$([M],[\tau(M)])$ for all $M\in\cP_m'$ as directed edges on the set of vertices~$\cP_m/{\sim}$; see Figure~\ref{fig:sptree}.
By definition, every vertex has exactly one outgoing edge, which rules out the existence of mixed cycles.
We will argue in the following that
\begin{equation}
\label{eq:phi-decrease}
\varphi(M)>\varphi(\tau(M))
\end{equation}
holds for all $M\in\cP_m'$, where the comparison operation~$>$ denotes lexicographic order.
This rules out the existence of oriented cycles, and proves the existence of a path to the unique equivalence class~$[M]$ with $\varphi(M)=(0,0,0,0,0)$, namely the class~$[M^1]$.
From these observations it follows that the directed edges~$([M],[\tau(M)])$ form a spanning tree.

\begin{figure}[t!]
\includegraphics[scale=0.7]{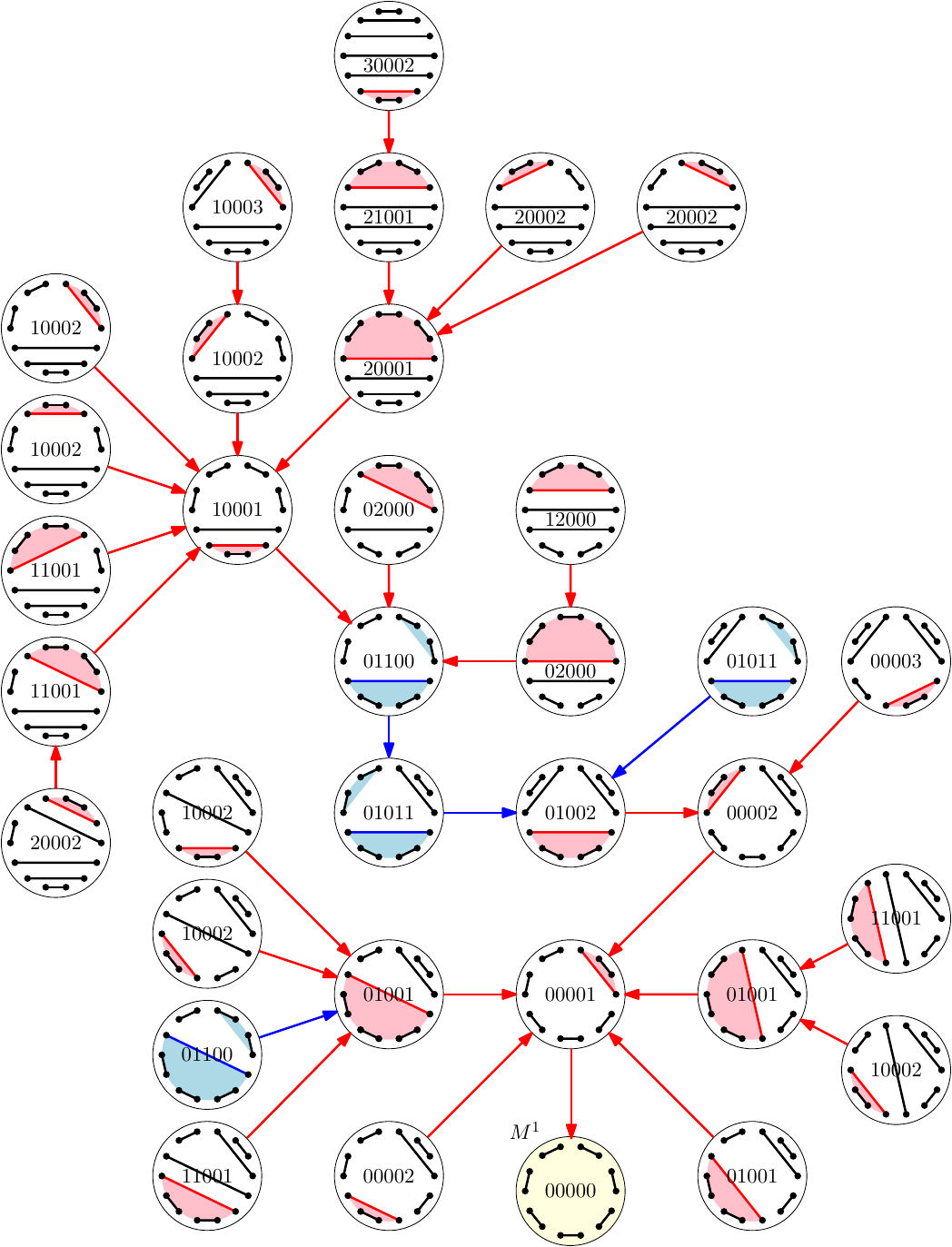}
\caption{Illustration of Lemma~\ref{lem:sptree} for $m=7$.
The red edges are closed ears, while the blue edges are open or half-open ears.
The 5-digit labels are as defined in~\eqref{eq:phi}, and they decrease lexicographically along directed edges of the tree.
The matching~$M^1$ with $\varphi(M^1)=00000$ is highlighted at the bottom.}
\label{fig:sptree}
\end{figure}

To prove~\eqref{eq:phi-decrease}, we distinguish the three cases for the ear~$e=(i,d)$ of maximum length chosen in the definition~\eqref{eq:tau}.
In case~\eqref{eq:tau1}, we either have $d=3$ or $d\geq 5$ and all such ears are closed.
In the former case, we have $\varphi(M)=(0,0,0,0,n_5)$ and $\varphi(M-e)=(0,0,0,0,n_5-1)$, so~\eqref{eq:phi-decrease} is proved.
In the latter case, we have $\varphi(M)=(n_1,n_2,0,0,n_5)$, as no open or half-open ears are present.
Removing the ear~$e$ in~$M$ to obtain~$M-e$ may turn a non-ear edge of length at least~5 into an ear edge, thus decreasing~$n_1$, or not, in which case $n_1$ is unchanged and $n_2$ decreases.
In both cases we have a lexicographic decrease.

In case~\eqref{eq:tau2}, we have $d\geq 5$ and the ear~$e$ is open or half-open, namely $\{i-4,i-3\},\{i-2,i-1\}\in M$.
Replacing the length~1 edges~$\{i-4,i-3\},\{i-2,i-1\}$ by $\{i-4,i-1\},\{i-3,i-2\}$ creates an additional ear edge of length~3 (increasing~$n_5$ by~1), and does not change the number of ear or non-ear edges of length at least~5 (so $n_1$ and~$n_2$ do not change).
However, it changes the status of the ear~$e$ from open to half-open or from half-open to closed, and possibly also for another ear (from an ear edge incident to the point~$i-5$ or~$i-7$), thus decreasing $n_3$ or~$n_4$, again a lexicographic decrease.

The argument in case~\eqref{eq:tau3} is symmetric.

This completes the proof of the lemma.
\end{proof}

\subsubsection{Edge-disjointness of the gluing cycles}
\label{sec:gluing}

By Lemma~\ref{lem:sptree}, the cycles of the factor~$\cC_m$ may be glued together to a single cycle (losing some vertices in the process) using suitable gluing cycles, as described by the spanning tree defined by~$\tau$.
In the next step, we need to make sure that these gluing cycles are edge-disjoint, so that the gluing operations do not interfere with each other.

For this we define, for any perfect matching~$M\in\cP_m'\cup\{M^1\}$, a set of intervals~$I(M)$ of integers.
Each interval in~$I(M)$ describes a subpath of the cycle~$C(M)$ (by a sequence of consecutive vertices) whose edges will be deleted when taking the symmetric difference with the gluing cycles that join~$C(M)$ to the other cycles of the factor~$\cC_m$, as described by the spanning tree.
Recall that for a gluing cycle~$D(M,e=(i,d))$, these intervals on the cycles~$C(M)$ and~$C(M-e)$ are $[i+1,i+d]$ and~$[i,i+d+1]$, respectively (see Figure~\ref{fig:ear}~(d)).
We choose these intervals so that their starting points are minimal ($\geq 1$).

Formally, the set~$I(M)$ is defined as follows.
Let $\Gamma(M)$ be the set of representatives of equivalence classes that are neighbors of~$[M]$ in the spanning tree, i.e., $\Gamma(M)\coloneq \{M'\in\cP_m'\mid [\tau(M')]=[M] \text{ or } [\tau(M)]=[M']\}$.

For each $M'\in\Gamma(M)$, we distinguish two cases:

Case~1 (incoming edges in the spanning tree): $[\tau(M')]=[M]$.

Case~1a: If $\tau(M')=M'-e'$ for some ear~$e'$ of~$M'$ (this is case~\eqref{eq:tau1} in the definition of~$\tau$), then we let $M''\in[M']$ and the corresponding ear~$e''=(i,d)$ of~$M''$ be such that $M''-e''=M$ and $i\in[1,\ord(M)]$, and we add $[i,i+d+1]$ to the set~$I(M)$.

Case~1b: If $[M-f]=[M']$ for some ear~$f=(i,3)$ of~$M$ (cases~\eqref{eq:tau2} and~\eqref{eq:tau3}), we choose this ear so that~$i\in[1,\ord(M)]$, and we add $[i+1,i+3]$ to the set~$I(M)$. 

Case~2 (the outgoing edge in the spanning tree): $[\tau(M)]=[M']$.

Case~2a: If $\tau(M)=M-e$ for some ear~$e=(i,d)$ of~$M$, we choose this ear so that~$i\in[1,\ord(M)]$, and we add~$[i+1,i+d]$ to the set~$I(M)$.

Case~2b: If $[M'-f']=[M]$ for some ear~$f'$ of length~3 of~$M'$, then we let $M''\in[M']$ and the corresponding ear~$f''=(i,3)$ of~$M''$ be such that $M''-f''=M$ and $i\in[1,\ord(M)]$, and we add $[i,i+4]$ to the set~$I(M)$.

\begin{lemma}
\label{lem:intervals}
We have $I(M^1)=\{[1,5]\}$.
Furthermore, for any $M\in\cP_m'$ each interval~$[i,j]\in I(M)$ satisfies $i\in [1,\ord(M)+1]$ and $j\in[1,2\ord(M)]$, and any pair~$[i,i+1]$ for $i\in[1,2\ord(M)-1]$ is contained in at most 7 intervals of~$I(M)$.
\end{lemma}

The proof of Lemma~\ref{lem:intervals} is illustrated in Figure~\ref{fig:intervals}.

\begin{figure}[t!]
\includegraphics{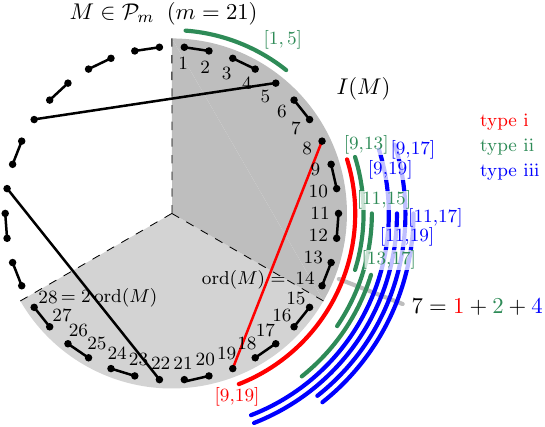}
\caption{Illustration of Lemma~\ref{lem:intervals}.}
\label{fig:intervals}
\end{figure}

\begin{proof}
We have $|\Gamma(M^1)|=1$, namely, there is exactly one matching~$M'\in\cP_m'$ with $[\tau(M')]=[M^1]$.
Its equivalence class contains the matching~$M''\coloneq \{\{1,4\},\{2,3\},\{5,6\},\ldots,\{n'-1,n'\}\}$ with the ear~$e''=(1,3)$ (case~1a), resulting in~$I(M)=\{[1,1+3+1]\}=\{[1,5]\}$.

For the rest of the proof consider a perfect matching~$M\in\cP_m'$.
Each interval in~$I(M)$ is associated with one of the following configurations in~$M$:
\begin{itemize}[topsep=1mm,leftmargin=4mm]
\item Type~i (from cases~1b and~2a of the definition): An ear~$(i,d)$ of~$M$ with $i\in[1,\ord(M)]$ contributes the interval~$[i+1,i+d]$.
\item Type~ii (from case~2b and case~1a when $d=3$): Two consecutive 1-edges~$\{i,i+1\},\{i+2,i+3\}$ in~$M$ with $i\in[1,\ord(M)]$ contribute the interval~$[i,i+4]$.
\item Type~iii (from case~1a when $d\geq 5$): Three or more consecutive 1-edges~$\{i,i+1\},\{i+2,i+3\},\ldots,\{j-1,j\}$, with $i\in[1,\ord(M)]$ with the extra condition that $\{i-4,i-3\},\{i-2,i-1\}\notin M$ and~$\{j+1,j+2\},\{j+3,j+4\}\notin M$ (recall from~\eqref{eq:tau1} the requirement for the ear of length at least~5 to be closed), contribute the interval~$[i,j+1]$.
\end{itemize}

Clearly, for each of the three types of intervals, the starting point of the interval lies within $[1,\ord(M)+1]$.
For type~i the ending point is~$i+d$ where $i\in[1,\ord(M)]$.
We will argue that $d<\ord(M)$, which implies that $i+d\leq 2\ord(M)$.
Indeed, if $\ord(M)=n'$ we have $d\leq m=n'/2<\ord(M)$, and if $\ord(M)<n'$ we also have $d<\ord(M)$, otherwise the edge~$\{i,i+d\}$ in~$M$ and the rotated edge~$\{i+\ord(M),i+d+\ord(M)\}$ would cross.
For type~ii the ending point is~$i+4$ where $i\in[1,\ord(M)]$.
As $M\notin[M^1]$ we have $\ord(M)\geq 3$.
Furthermore, the case $\ord(M)=3$ is ruled out by the presence of the edges~$\{i,i+1\},\{i+2,i+3\}$.
Consequently, we have $\ord(M)\geq 4$ and thus $i+4\leq 2\ord(M)$.
For type~iii the ending point is~$j+1$ and there are 1-edges on all points~$i,i+1,\ldots,j$ where $i\in[1,\ord(M)]$.
As $M$ has not only 1-edges, we must have $j<2\ord(M)$ and thus $j+1\leq 2\ord(M)$.

Consider any pair~$[i,i+1]$ with $i\in[1,2\ord(M)-1]$.
It is easy to check the following; see Figure~\ref{fig:intervals}:
At most one interval of type~i contains this pair.
At most two intervals of type~ii contain this pair.
At most four intervals of type~iii contain this pair.
We thus have at most 1+2+4=7 containments.
\end{proof}

\subsubsection{Proof of Theorem~\ref{thm:GO-long}}

\begin{proof}[Proof of Theorem~\ref{thm:GO-long}]
To construct a long cycle in the flip graph, we start with the cycle factor~$\cC_m$ defined in~\eqref{eq:cfac}.
We glue these cycles together, by taking the symmetric difference with suitable pairwise edge-disjoint gluing cycles, along the spanning tree defined by the mapping~$\tau$ from~\eqref{eq:tau} (recall Lemma~\ref{lem:sptree}), losing some vertices in the process.
In the end we bound the number of vertices that are not included in the cycle.

For $m\leq 6$, the statement of the theorem is void, so we will assume that $m\geq 7$, i.e., $n\geq 15$, for the rest of the proof.

\begin{figure}[t!]
\makebox[0cm]{ 
\includegraphics[page=2,scale=0.55]{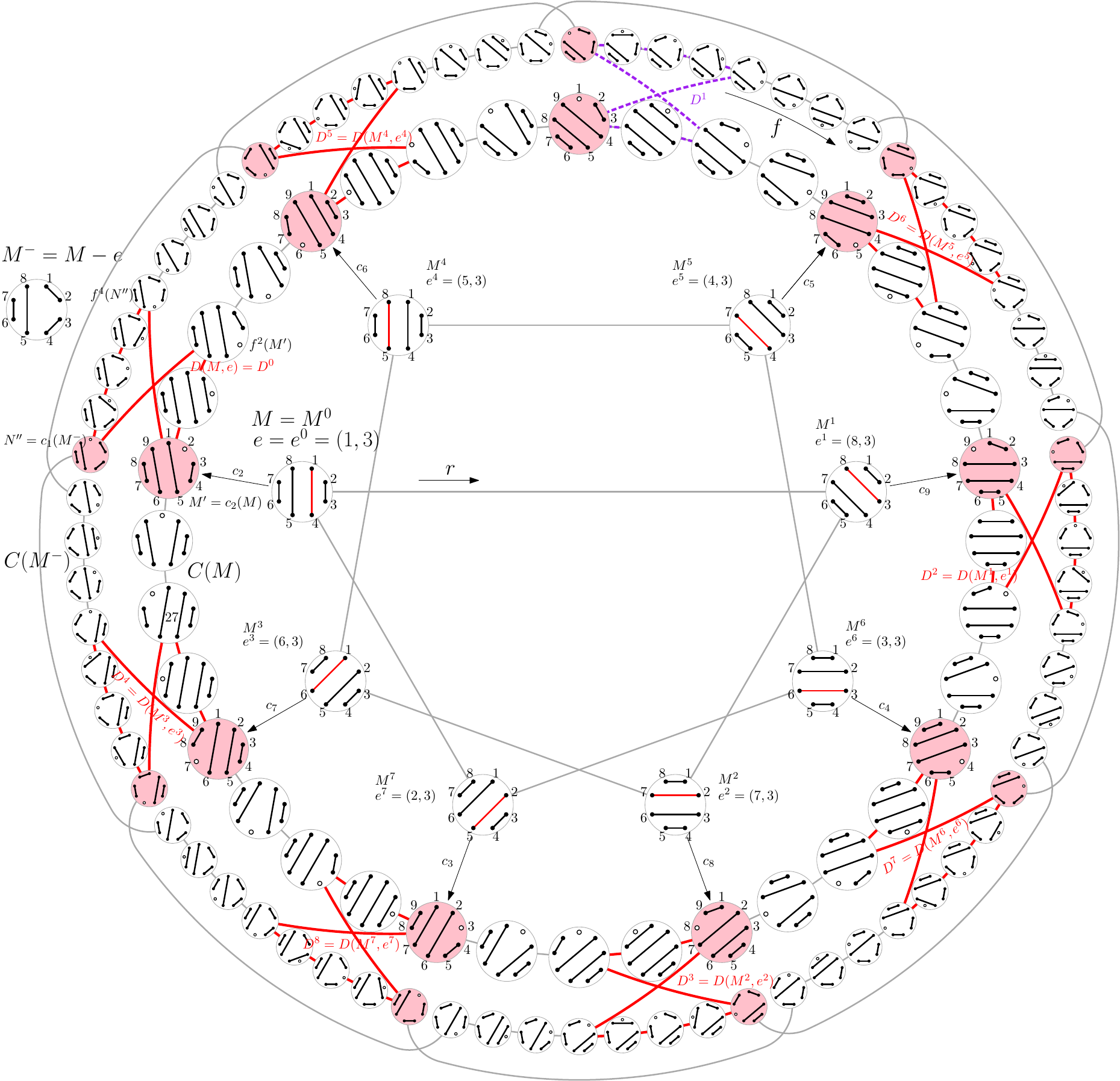}
}
\caption{Shifting of gluing cycles applied in the proof of Theorem~\ref{thm:GO-long}.
The example continues the one from Figure~\ref{fig:cfac}, showing the first two cycles from the factor displayed there at the outer perimeter, and the gluing cycles between them (in red).
All 9 cyclic rotations of one matching on each of the two cycles are highlighted.}
\label{fig:shift}
\end{figure}

Recall Figure~\ref{fig:ear} for the following definitions.
Consider a perfect matching $M\in\cP_m'$ with an ear~$e=(i,d)$, and define $M'\coloneq c_{i+1}(M)$, $M^-\coloneq M-e$ and $N''\coloneq c_i(M^-)$.
The gluing cycle~$D(M,e)$ intersects with~$C(M)$ in the vertices~$f^s(M')$ for~$s=0,\ldots,d-1$, and it intersects with~$C(M^-)$ in the vertices~$f^s(N'')$ for $s=0,\ldots,d+1$.
As mentioned before, taking the symmetric difference of the cycle~$D(M,e)$ with~$C(M)$ and~$C(M^-)$ joins those two cycles together, at the cost of losing~$2d-2$ vertices.
The crucial observation is that there are actually $n$ distinct cycles that achieve the same gluing effect, obtained by cyclically rotating the involved matchings.
Specifically, we let $D^j\coloneq r^j(D(M,e))$ for $j=0,\ldots,n-1$ be the cycle obtained from~$D(M,e)$ by rotating all involved matchings by $j$ steps; see Figure~\ref{fig:shift}.
Note that from these $n$ sequences of matchings on~$C(M)$ (and $C(M^-)$, respectively), $n-1=n'$ many correspond to gluing cycles~$D(M^j,e^j)$ where $(M^j,e^j)\coloneq r^j(M,e)$ for $j=0,\ldots,n-2$, but one of them does not (drawn dashed in the figure), namely the one whose first matching has the vertex~1 (vertex~$n$) unmatched.
Then the cycle~$D^j$ intersects with~$C(M)$ in the vertices~$f^{j\cdot (n-1)+s}(M')$ for~$s=0,\ldots,d-1$, and it intersects with~$C(M^-)$ in the vertices~$f^{j\cdot (n-1)+s}(N'')$ for $s=0,\ldots,d+1$.
As the cycle~$C(M)$ has length~$n\ord(M)$ (recall Lemma~\ref{lem:cfac}), these are exactly the vertices~$f^{k\cdot\ord(M)+s}(M')$ for $s=0,\ldots,d-1$ and some $k\in\{0,\ldots,n-1\}$ (and $k$ ranges over all values from this set, as $j$ ranges from $0,\ldots,n-1$).
Similarly, as $C(M^-)$ has length~$n\ord(M^-)$, these are exactly the vertices~$f^{k\cdot\ord(M^-)+s}(N'')$ for $s=0,\ldots,d-1$ and some $k\in\{0,\ldots,n-1\}$.
Therefore, selecting one of the cycles~$D^j$, $j=0,\ldots,n-1$, to glue the cycles~$C(M)$ and~$C(M^-)$ together corresponds to shifting the subpath in which it intersects with~$C(M)$ and with~$C(M^-)$, respectively, by multiples of~$\ord(M)$ or~$\ord(M^-)$, respectively.

For any perfect matching~$M\in\cP_m'$, one copy of the subpaths in which gluing cycles intersect with~$C(M)$ is recorded by the set of intervals~$I(M)$ defined in Section~\ref{sec:gluing}.
By Lemma~\ref{lem:intervals}, whenever two such subpaths share an edge (which does not happen if $M=M^1$), applying a shift by a multiple of~$2\ord(M)$ makes them edge-disjoint.
By the same lemma, at most 7 subpaths overlap, and since $n\geq 15$ we have that the cycle length satisfies $n \ord(M)\geq 7\cdot 2\ord(M)=14\ord(M)$, so all subpaths can be shifted appropriately to be edge-disjoint.
Note that these shifting amounts are fixed cycle after cycle, following an arbitrary traversal of the spanning tree defined by~$\tau$, so that at each step, at most one subpath is fixed by previous decisions, and all others can be shifted freely.
We write~$L$ for the resulting cycle, obtained by taking the symmetric difference of the edge sets of the cycle factor~$\cC_m$, and the appropriately shifted (and thus pairwise edge-disjoint) gluing cycles.

To complete the proof, we bound the number of vertices in the flip graph not contained in~$L$.
For this, note that for any perfect matching~$M\in\cP_m'$, by Lemma~\ref{lem:intervals}, the total length of all intervals in~$I(M)$ is at most $7\cdot 2\ord(M)=14\ord(M)$.
Consequently, the total number of vertices from the cycle~$C(M)$ not included in~$L$ is at most $14\ord(M)$.
As the length of the cycle~$C(M)$ is~$n\ord(M)$, the cycle~$L$ misses at most a $14/n$-fraction of vertices from the cycle~$C(M)$.
As this bound applies uniformly to every cycle from the factor, it also applies to the entire cycle~$L$.

The proof of the theorem is complete.
\end{proof}

\torsten{The bound $14/n$ can probably be improved. Currently the bound is established for every cycle. However, looking at Figure~\ref{fig:sptree} it seems to be much better to do the estimation over all edges of the entire spanning tree (because in most cases, the joining cycle is short/constant, and the cycle from the factor has quadratic length (as $\ord(M)=n'$ for most matchings) so only much fewer vertices are missed). Can we thus prove only a $\Theta(1/n^2)$-fraction of missed vertices?}

\subsubsection{Proof of Theorem~\ref{thm:GO-miss}}

This proof reuses some of the terminology introduced in the previous sections, specifically the definition of the mapping~$f$ and of the cycles~$C(M)$ for perfect matchings~$M\in \cP_m$ given after Lemma~\ref{lem:cfac}.

\begin{proof}[Proof of Theorem~\ref{thm:GO-miss}]
We refer to a pair of edges~$\{i,i+3\},\{i+1,i+2\}$ of lengths~3 and~1 as a \defi{twin}.
For any bitstring $\alpha\in\{0,1\}^a$, we define a perfect matching~$M_\alpha\in\cP_m$ as follows; see Figure~\ref{fig:twin1}:
\begin{enumerate}[topsep=1mm,leftmargin=7mm,label=(\arabic*)]
\item on the first 18 points $1,\ldots,18$, first place two 1-edges, then a twin, then a 1-edge, then a twin, then two 1-edges (blue edges in the figure);
\item for each $i=1,\ldots,a$, on the next 10 points $19+10(i-1),\ldots,28+10(i-1)$, place a 1-edge, then a twin that is followed by another 1-edge, if $\alpha_i=0$, or preceded by another 1-edge if $\alpha_i=1$, and then another 1-edge (red edges in the figure);
\item on the remaining $2b$ points $19+10a,\ldots,18+10a+2b$, place $b$ many 1-edges (green edges in the figure).
\end{enumerate}
Note that the parameter~$\alpha$ controls the placement of the twins in~$M_\alpha$ added in step~(2).
On the other hand, the twins added in step~(1) are independent of~$\alpha$, and they serve as a unique reference to ensure disjointness.
The $b\leq 4$ many 1-edges added in step~(3) are simply used to fill up the remaining pairs of points that are not enough to create another block of 5 pairs in the second step. 

\begin{figure}[t!]
\includegraphics[page=1,scale=0.7]{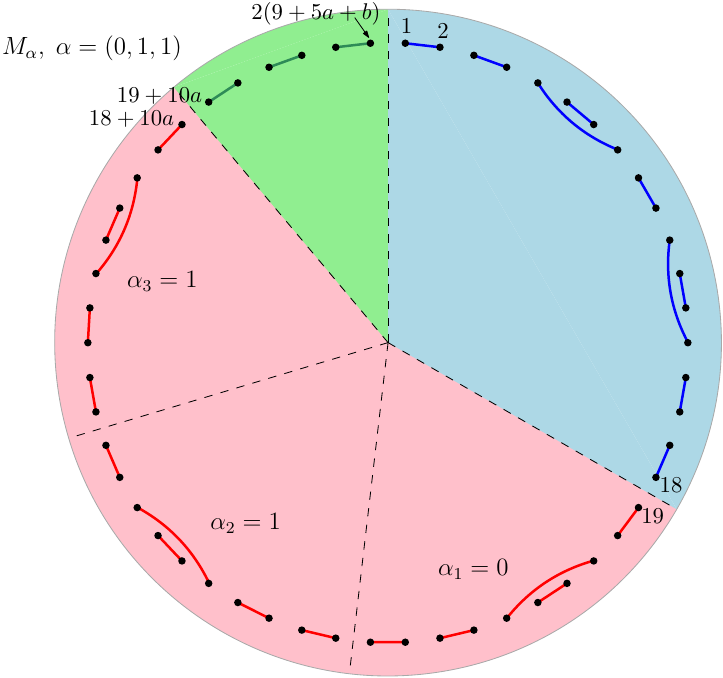}
\caption{Definition of the matchings~$M_\alpha$.
The three colors correspond to the three steps of the construction.}
\label{fig:twin1}
\end{figure}

For any matching~$M_\alpha$, let $M'$ be a matching obtained from~$M_\alpha$ by replacing one of its twins by two 1-edges, an operation that we refer to as \defi{twin removal}.
Using that the two twins added in step~(1) are the unique twins separated by only a single 1-edge, one can check the following:
Let $\alpha,\beta\in\{0,1\}^a$ be such that $\alpha\neq \beta$, and let $M_\alpha'$ and~$M_\beta'$ be obtained from~$M_\alpha$ and~$M_\beta$, respectively, by a twin removal.
Then we have
\begin{subequations}
\label{eq:diff-cycles}
\begin{equation}
[M_\alpha]\neq [M_\beta], [M_\alpha]\neq [M_\beta'], \text{ and } [M_\alpha']\neq [M_\beta].
\end{equation}
Furthermore, we have
\begin{equation}
[M_\alpha']\neq [M_\beta'], \text{ unless}
\end{equation}
\end{subequations}
$\alpha$ and $\beta$ differ in a single bit, and $M_\alpha'$ and~$M_\beta'$ are obtained from~$M_\alpha$ and~$M_\beta$, respectively, by removing the corresponding twin in which the two matchings differ.

\begin{figure}[t!]
\makebox[0cm]{ 
\includegraphics[page=2,scale=0.8]{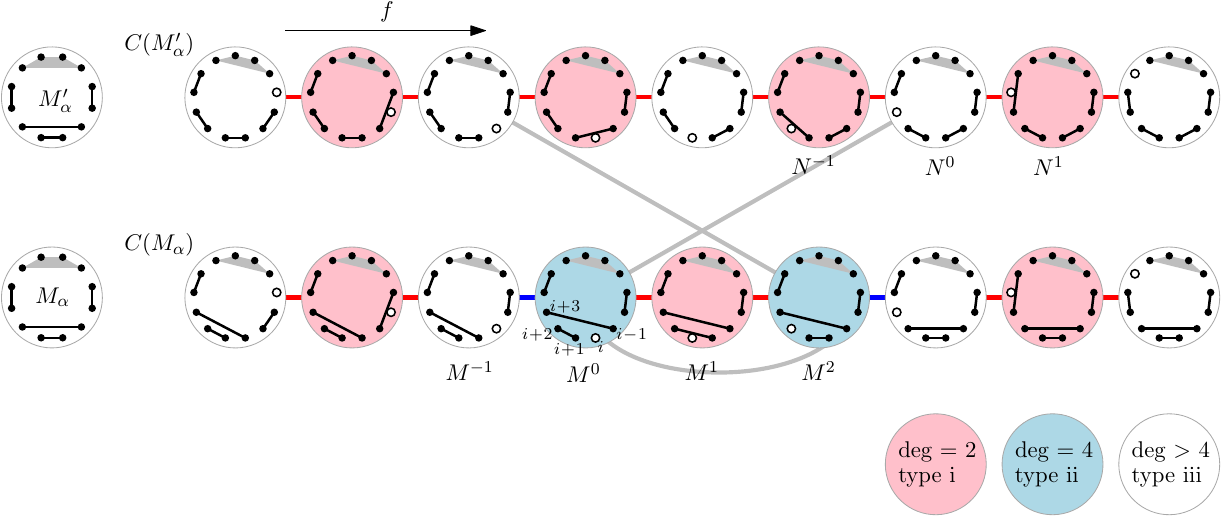}
}
\caption{Illustration of the proof of Theorem~\ref{thm:GO-miss}.
The gray areas represent an arbitrary fixed matching on the remaining points.
We argue in the proof that the gray edges in the flip graph cannot be used, and thus the blue edges must be used.}
\label{fig:twin2}
\end{figure}

For any perfect matching~$M_\alpha$, consider the cycle~$C(M_\alpha)$.
Each matching~$M'\in C(M_\alpha)$ is of one of the following three types; see Figure~\ref{fig:twin2}:
\begin{itemize}[topsep=1mm,leftmargin=4mm]
\item Type~i: $M'$ has a 2-edge $\{i-1,i+1\}$, where $i$ is the unmatched point, in which case we have $\deg(M')=2$ in the flip graph.
\item Type~ii: $M'$ has no 2-edge but a 4-edge $\{i-1,i+3\}$, where $i$ or $i+2$ is the unmatched point, in which case we have $\deg(M')=4$.
\item Type~iii: $M'$ has neither 2- nor 4-edges, then both neighbors of~$M'$ on~$C(M_\alpha)$ are of type~i or~ii.
\end{itemize}

We consider a matching~$M^0\in C(M_\alpha)$ of type~ii with a 4-edge $\{i-1,i+3\}$, and assume w.l.o.g.\ that the unmatched point is~$i$.
Note that from point~$i$, the 4-edge~$\{i-1,i+3\}$ and the 1-edge~$\{i+1,i+2\}$ are visible.
Then $N^0\coloneq u(M^0,i+3)$ is one of the four neighbors of~$M^0$ and lies on the cycle~$C(M_\alpha')$ where $M_\alpha'$ is obtained from~$M_\alpha$ by removing one of its twins.
The other three neighbors of~$M^0$ are $M^{-1}\coloneq f^{-1}(M^0)=u(M^0,i-1)$, $M^1\coloneq f(M^0)=u(M^0,i+1)$ and~$M^2\coloneq f^2(M^0)=u(M^0,i+2)$.

We observe the following: Any path~$P$ in the flip graph that has its first and last vertex neither on~$C(M_\alpha)$ nor on~$C(M_\alpha')$ and contains all vertices of these two cycles (in arbitrary order, i.e., not necessarily in the order along the cycles), must not use the edges~$\{M^0,N^0\}$ and~$\{M^0,M^2\}$.
Indeed, the edge~$\{M^0,N^0\}$ cannot be used, as both of the two neighboring vertices~$N^{-1}\coloneq f^{-1}(N^0)$ and~$N^1\coloneq f(N^0)$ on the cycle~$C(M_\alpha')$ have degree~2 (any twin is preceded and followed by at least one 1-edge) and thus each of their two incident edges must be used.
The edge~$\{M^0,M^2\}$ cannot be used, as the neighboring vertex~$M^1$ on the cycle~$C(M_\alpha)$ has degree~2 and thus its two incident edges must be used, and the edge~$\{M^0,M^2\}$ would create a triangle.

Applying this argument for all vertices of type~ii on~$C(M_\alpha)$ would imply that~$P$ must contain all edges of the cycle~$C(M_\alpha)$, which is impossible.
We conclude that any path~$P$ in the flip graph that has its first and last vertex neither on~$C(M_\alpha)$ nor on~$C(M_\alpha')$ for any of the matchings~$M'$ obtained from~$M_\alpha$ by twin removal misses at least one vertex from one of these cycles.

Using~\eqref{eq:diff-cycles}, and the fact there are $2^a=2^{(m-b-9)/5}$ choices for $\alpha\in\{0,1\}^a$, we conclude that any path misses at least $(2^a-4)/2=2^{a-1}-2=2^{(m-b-14)/5}-2\geq 2^{(m-18)/5}-2\geq 2^{m/5}/16$ many vertices.
\end{proof}

\subsubsection{General point sets}

For point sets in general position, the bounds provided by Theorem~\ref{thm:GO-miss} can be improved, and the arguments simplified considerably.

\begin{figure}[t!]
\includegraphics[page=2,scale=0.7]{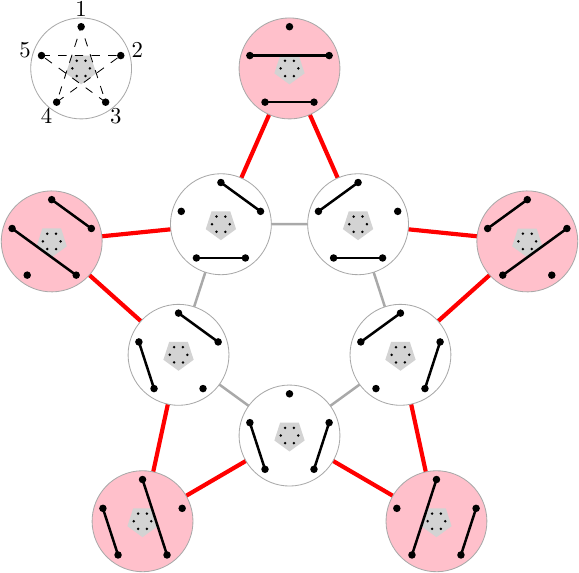}
\caption{10-cycle in flip graph of non-crossing matchings on 11 points induced by 5 degree-2 vertices and their incident edges (cf. Figure~\ref{fig:G5geom}).
On the inner 6 points, an arbitrary perfect matching can be prescribed.}
\label{fig:G11geom}
\end{figure}

\begin{theorem}
\label{thm:GOP}
Let $m\geq 5$ and $n\coloneq 2m+1$.
There is a set of $n$ points such that in the flip graph of non-crossing almost-perfect matchings under 1-flips, any path misses at least $C_{m-2}-2=\Theta(4^m/m^{3/2})$ many vertices.
In particular, there is no Hamilton path (nor cycle).
\end{theorem}

\begin{proof}
We place 5 points~$1,2,3,4,5$ in convex position, and the remaining $n-5$ points in convex position within the 5-gon that sits inside each of the triangles formed by three outer points~$i,i+1,i+3$ for $i=1,\ldots,5$ (modulo~5); see Figure~\ref{fig:G11geom}.
Then for any fixed choice of a perfect matching on the inner~$n-5$ points, there are 10 ways to match the outer 5 points (one of them remains unmatched), and among those 10 matchings, 5 have degree~2 in the flip graph, so together with their incident edges they form a 10-cycle in the flip graph.
The number of those 10-cycles is given by the number of perfect matchings on~$n-5$ points, namely $C_{m-2}$.
Clearly, any path misses at least one vertex from all but at most two of these cycles.
\end{proof}

\bibliographystyle{alpha}
\bibliography{refs}

\end{document}